\documentclass[a4paper,11pt]{amsart}
\usepackage[utf8]{inputenc}
\usepackage[english]{babel}
\usepackage[bitstream-charter]{mathdesign}
\usepackage[T1]{fontenc}
\usepackage{amsfonts}
\usepackage{geometry}
\usepackage{amsmath}
\usepackage[colorlinks = true]{hyperref}
\usepackage{color}
\usepackage{tikz}
\usepackage{pgf}
\usepackage{mathrsfs}
\usetikzlibrary{arrows}
\usepackage{amsmath,amsthm}
\usepackage{cancel}
\usepackage{wrapfig}
\usetikzlibrary{arrows,matrix,positioning}
\usepackage{graphicx}
\usepackage{multirow}
\usepackage{arydshln}

\newcommand{\CC}{\mathbb{C}}
\newcommand{\NN}{\mathbb{N}}
\newcommand{\PP}{\mathbb{P}}
\newcommand{\XX}{\mathbb{X}}
\newcommand{\AAA}{\mathbb{A}}

\newcommand{\HF}{\mathrm{HF}}
\newcommand{\caL}{\mathcal{L}}
\newcommand{\caS}{\mathcal{S}}

\theoremstyle{plain}
\newtheorem{theorem}{Theorem}[section]
\newtheorem{proposition}[theorem]{Proposition}   
\newtheorem{corollary}[theorem]{Corollary}        
\newtheorem{lemma}[theorem]{Lemma}

\theoremstyle{definition}              
\newtheorem{definition}[theorem]{Definition}

\newtheorem{example}[theorem]{Example}

\theoremstyle{remark}
\newtheorem*{notation}{{\it Notation}}
\newtheorem{remark}[theorem]{Remark}
\title{On the Hilbert function of general fat points in $\PP^1 \times \PP^1$}
\author[E. Carlini]{Enrico Carlini}
\address[E. Carlini]{DISMA Department of Mathematical Sciences, Politecnico di Torino, Turin, Italy}
\email{enrico.carlini@polito.it}

\author[M. V. Catalisano]{Maria Virginia Catalisano}
\address[M. V. Catalisano]{Dipartimento di Ingegneria Meccanica, Energetica, Gestionale e dei Trasporti, Universit\`a degli studi di Genova, Genoa, Italy}
\email{catalisano@diptem.unige.it}

\author[A. Oneto]{Alessandro Oneto}
\address[A. Oneto]{INRIA Sophia Antipolis M\'editerran\'ee (team Aromath), Sophia Antipolis, France}
\email{alessandro.oneto@inria.fr}


\begin{document}
\maketitle

\begin{abstract}
 We study the bi-graded Hilbert function of ideals of general fat points with same multiplicity in $\PP^1\times \PP^1$. Our first tool is the multiprojective-affine-projective method introduced by the second author in previous works with A.V.~Geramita and A.~Gimigliano where they solved the case of double points. In this way, we compute the Hilbert function when the smallest entry of the bi-degree is at most the multiplicity of the points. Our second tool is the differential Horace method introduced by J.~Alexander and A.~Hirschowitz to study the Hilbert function of sets of fat points in standard projective spaces. In this way, we compute the entire bi-graded Hilbert function in the case of triple points.
\end{abstract}

\section{Introduction}
Problems regarding polynomial interpolation are very classical and have been studied since the beginning of last century. In the classic case, we consider a set of points on the projective plane and we want to compute the dimension of the \textcolor{black}{linear system} of curves of given degree passing through the set of points with prescribed multiplicity. Sometimes, the linear conditions imposed by the the multiple points on the space of curves of given degree are not always independent. In these cases, we say that we have {\it unexpected curves}. In the case of less than $8$ points, these cases were explained by G.~Castelnuovo \cite{Ca91} and, more recently, in the work of M.~Nagata \cite{Na60}. In general, the situation is completely described by the famous {\it SHGH Conjecture} which takes the name after the works of B.~Segre \cite{Segre61}, B.~Harbourne \cite{Harbourne85}, A.~Gimigliano \cite{Gimigliano88} and A.~Hirschowitz \cite{Hirschowitz89}.

In the language of modern commutative algebra, this means to compute the Hilbert function of the ideal of {\it fat points} with prescribed multiplicity and look at its Hilbert function in a given degree. By parameter count, we expect that such a dimension is the difference between the dimension of the space of curves of the fixed degree and the \textcolor{black}{number of} conditions imposed by the fat points. If this is the \textcolor{black}{actual} dimension, we say that the set of fat points impose {\it independent conditions} on the linear system of curves of the fixed degree.

In this paper, we want to consider a multi-graded interpolation problem. We consider ideals of fat points in $\PP^1\times\PP^1$ with support in general position and we look at its bi-graded Hilbert function. The case of double points has been settled by the second author together with A.V.~Geramita and A.~Gimigliano \cite{CGG05}. They introduced a method called {\it multiprojective-affine-projective method} which allows to reduce the multi-graded problem to a question in the standard projective plane.

A milestone for polynomial interpolation problems is the work \textcolor{black}{by} J.~Alexander and A.~Hirschowitz \textcolor{black}{in which the authors considered} ideals of double points in general position in the standard $n$-dimensional projective spaces. Exceptional cases where ideals of double points fail to give independent conditions on hypersurfaces of some degree were known since the beginning of last century, but we had to wait until 1995 for a complete classification which is now the so-called Alexander-Hirschowitz Theorem. The complete proof came after a series of enlightening papers where they developed a completely new method of approach called {\it m\'ethode d'Horace diff\'erentielle} \cite{AHa,AHb,AH,AH00}. 

\textcolor{black}{In our computations, we use the multiprojective-affine-projective method of \cite{CGG05} to reduce the problem to the study of fat points in $\PP^2$, where we use the differential Horace method of \cite{AH00}.}

\subsubsection*{Formulation of the problem and main results.}
\noindent Let $S = \CC[x_0,x_1;y_0,y_1] = \bigoplus_{i,j} S_{i,j}$ be the bi-graded coordinate ring of $\PP^1\times\PP^1$, namely $S_{i,j}$ is the vector space of bi-homogeneous polynomials of bi-degree $(i,j)$.

\begin{definition}
Let $\{P_1,\ldots,P_s\}$ be a set of points in $\PP^1\times\PP^1$. We will always assume, except when explicitly mentioned, that they are in general position. Let $\wp_i\subset S$ be the prime ideal defining the point $P_i$, respectively. The {\it scheme of fat points} of {\it multiplicity $m\geq 1$} with {\it support} at the $P_i$'s is the $0$-dimensional scheme $\XX$ defined by the ideal $I_{\XX,\PP^1\times\PP^1} = \wp_1^m \cap \ldots \cap \wp_s^m$. \textcolor{black}{If there is no ambiguity, we simply denote it by $I_{\XX}$.}
\end{definition}

For any bi-homogeneous ideal $I$ in $S$, we define the {\it Hilbert function} of $S/I$ as
$$
 \HF_{S/I}(a,b) := \dim_{\CC}(S/I)_{(a,b)} = \dim_\CC S_{(a,b)} - \dim_{\CC} I_{(a,b)}, \text{ for } (a,b) \in \NN^2.
$$
For short, we denote by $\HF_{\XX}$ the Hilbert function of the quotient ring $S/I_{\XX}$.
\begin{center}
\medskip
{\sc Question.} {\it Let $\XX$ be a scheme of fat points of multiplicity $m$ in $\PP^1\times \PP^1$. 

What is the bi-graded Hilbert function of $\XX$?}
\medskip
\end{center}

It is well-known that if $a,b \gg 0$, then the Hilbert function stabilize and is equal to degree of the scheme $\XX$, i.e., if $\XX$ is a scheme of $s$ fat points of multiplicity $m$, $\HF_{\XX}(a,b) = s{m+1 \choose 2}$. 

 In this paper, we study the Hilbert function in the case of {\it general} fat points of multiplicity $m$ in $\PP^1 \times \PP^1$. E. Guardo and A. Van Tuyl give a bound for the region \textcolor{black}{where} the Hilbert function of {\it general} fat points \textcolor{black}{becomes constant}, see \cite{GVT}. Our first result is \textcolor{black}{Theorem \ref{thm:main_any_m}}, where we compute the Hilbert function for {\it low} bi-degrees, namely for bi-degrees $(a,b)$ such that $\min\{a,b\} \leq m$. 

\setcounter{theorem}{9}
\setcounter{section}{3}
  \begin{theorem}\label{thm:main_any_m}
  Let $a \geq b$ and assume $m \geq b$. Let $\XX = mP_1 + \ldots + mP_s \subset \PP^1 \times\PP^1$. Then, 
 $$
  \HF_\XX(a,b) = \min\left\{(a+1)(b+1),
  	s{m+1 \choose 2} - s{m-b \choose 2}\right\},
 $$ 
 except if $s = 2k+1$ and $a = bk+c+s(m-b)$, with $c = 0,\ldots,b-2$, where 
 $$
 	\HF_\XX(a,b) = (a+1)(b+1) - {c+2 \choose 2}.
 $$
 \end{theorem}

\textcolor{black}{Observe that, in all our formulas, we use the standard rule that ${i \choose j} = 0$ if $i < j$.}

Since in each row and column the Hilbert function is eventually constant and Theorem \ref{thm:main_any_m} gives us the Hilbert function only for bi-degrees such that $\min\{a,b\} \leq m$, \textcolor{black}{in general} we are left with an intermediate region where we cannot conclude our computations (see Example \ref{example:HF}). \textcolor{black}{However, in case of triple points, we are able to give a complete description of the Hilbert function in Theorem \ref{thm:main_triple_multi}.}

\setcounter{section}{4}
\setcounter{theorem}{5}
\begin{theorem}\label{thm:main_triple_multi}
	Let $\XX = 3P_1 + \ldots + 3P_s \subset \PP^1 \times \PP^1$. Then,
	$$
		\HF_\XX(a,b) = \min\left\{(a+1)(b+1), 6s\right\},
	$$
	except for
	\begin{enumerate}
	\item $b=1$ and $s<\frac{2}{5}(a+1)$, where $\HF_{\XX}(a,1) = 5s$;
\item  $s$ odd, say $s = 2k+1$, and \\
	 	$\quad (a,b) = (4k+1,2)$, where $\HF_\XX(4k+1,2) = (a+1)(b+1) - 1$; \\
   		$\quad (a,b) = (3k,3)$, where $\HF_\XX(3k,3) = (a+1)(b+1) - 1$;\\
	 	$\quad (a,b) = (3k+1,3)$, where $\HF_\XX(3k+1,3) = 6s- 1$; \\
	\item $s = 5$ and $(a,b) = (5,4)$, where $\HF_X(5,4) = 29$.\\
\end{enumerate}
\end{theorem}

\setcounter{section}{1}

\begin{remark}
Recall that $\PP^1\times\PP^1$ is the {\it Hirzebruch surface} usually called $\mathbb{F}_0$, thus in this paper we study linear systems with multiple base points on $\mathbb{F}_0$. The general case of linear systems with multiple points on the Hirzebruch surfaces $\mathbb{F}_n$ has been considered in several papers. For example, Proposition 5.2 in \cite{La02} gives a list of exceptional cases for any natural number $n$. However, in the case $n = 0$, the exception (iii) of our Theorem \ref{thm:main_triple_multi} is missing. Moreover, we point out that the methods used in \cite{La02} are different from the ones we are using.
\end{remark}

\subsubsection*{Structure of the paper.} In Section \ref{sec:2}, we explain our approach and we describe the main tools we are going to use. In particular, we describe the {\it multiprojective-affine-projective method}, introduced by the \textcolor{black}{second} author together with A.V.~Geramita and A.~Gimigliano \cite{CGG05}, and the {\it differential Horace method} of J.~Alexander and A.~Hirschowitz \cite{AH00}. In Section \ref{sec:3}, we prove Theorem \ref{thm:main_any_m}. In Section \ref{sec:4}, by the differential Horace method we compute the complete bi-graded Hilbert function in the case of triple points (Theorem \ref{thm:main_triple_multi}). In the Appendix, we implement our results with the algebra software {\it Macaulay2} \cite{M2}. 

\subsubsection*{Acknowledgements.} The authors are grateful to R.~Fr\"oberg for suggesting the problem. The third author thanks all the participants of the {\it Problem Solving Seminar in Commutative Algebra} at the Department of Matematics of Stockholm University for useful discussions. The authors wish to thank A. Laface and A. Van Tuyl for useful comments. The first author acknowledges the finacial support of the Politecnico di Torino (Italy) through the ``Ricerca di Base" fund RDB\_17. The second author was supported  by MIUR funds (PRIN 2010-11 grant 2010S47ARA-004 -- Geometria delle Variet\`a Algebriche) (Italy) and by the Universit\`a degli Studi di Genova through the ``Progetti di Ateneo 2015".

\section{Approach and tools}\label{sec:2}
\subsection{Multiprojective-affine-projective method}
In \cite{CGG05}, the authors introduced a method to reduce the multi-graded problem to the graded case. We describe this method in the case of $\PP^1\times\PP^1$, which is the one of our interest. We consider the following birational map
\begin{center}
\begin{tabular}{c c c c c c}
 $\phi$: & $\PP^1 \times \PP^1$ & $\dashrightarrow$ & $ \AAA^2$ & $\longrightarrow$ & $\PP^2,$ \\
 & $([a_0:a_1],[b_0:b_1])$ & $\mapsto$ & $\left(\frac{a_1}{a_0},\frac{b_1}{b_0}\right)$ &$\mapsto$ & $\left[1:\frac{a_1}{a_0}:\frac{b_1}{b_0}\right]$. \\
\end{tabular}
\end{center}
This map is well-defined on the chart $\mathcal{U} = \{a_0b_0 \neq 0\}$.

 Given a set of fat points $X = m_1P_1+\ldots+m_sP_s$ in $\PP^n$ defined by the ideal $I_{X,\PP^n} = \wp_1^{m_1}\cap\ldots\cap\wp_s^{m_s}$, we denote by $\HF_X(d)$ the standard graded Hilbert function in degree $d$ of the corresponding quotient ring. In case of no ambiguity about the ambient space, we simply denote the ideal as $I_{X}$. %\textcolor{red}{For details, see...}

\begin{lemma}\label{lem:multiproj-aff-proj}\cite[Theorem 1.5]{CGG05}
 Let $a,b$ be positive integers and let $\XX$ be a $0$-dimensional scheme with support in $(\PP^1\times\PP^1)\setminus \mathcal{U}$. Let $X = aQ_1 + bQ_2 +\phi(\XX) \subset \PP^2$, where $Q_1 = [0:1:0]$, $Q_2 = [0:0:1]$. Then,
 $$
  \dim(I_{\XX,\PP^1\times\PP^1})_{(a,b)} = \dim(I_{X,\PP^2})_{a+b}.
 $$
\end{lemma}
\begin{notation}
	From now on, let $\XX$ be a scheme of $s$ fat points of multiplicity $m$ in $\PP^1\times\PP^1$, i.e., $$\XX = mP_1 + \ldots + mP_s;$$let $X$ be the scheme defined as in Lemma \ref{lem:multiproj-aff-proj}, i.e., $$X = aQ_1 + bQ_2 + mP_1 + \ldots + mP_s,$$ where, with an abuse of notation, we denote by $P_i$ both the point in $\PP^1\times \PP^1$ and its image in $\PP^2$. Moreover, let $\caL_{a,b}(\XX)$ be the linear system of curves in $\PP^1 \times \PP^1$ of bi-degree $(a,b)$ passing through $\XX$ and let $\caL_{d}(X)$ be the linear system of curves in $\PP^2$ of degree $d$ passing through $Z$.
\end{notation}
\noindent By Lemma \ref{lem:multiproj-aff-proj}, our original question is equivalent to the following.
\smallskip

\begin{center}
{\sc Question.} {\it Let $X = aQ_1 + bQ_2 + mP_1 + \ldots + mP_s \subset \PP^2$ be a scheme of fat points 

 in general position. What is the dimension of the linear system $\caL_{a+b}(X)$?
}
\end{center}

\smallskip

\noindent The {\bf virtual dimension} of the linear system $\caL_{a+b}(X)$, given by a parameter count, is 
\begin{align*}\label{exp.dim}
 \textit{vir}.\dim\caL_{a+b}(X) & = {a+b+2 \choose 2} - {a+1 \choose 2} - {b+1 \choose 2} - s{m+1 \choose 2},
\end{align*}
which coincides with the virtual dimension of the linear system $\caL_{a,b}(\XX)$
\begin{equation*}
	\textit{vir}.\dim\caL_{a,b}(\XX) = (a+1)(b+1) - s{m+1 \choose 2}.
\end{equation*}
The {\bf expected dimension} is defined as the maximum between $0$ and the virtual dimension. If the actual dimension is equal to the virtual dimension, we say that $X$ {\bf imposes independent conditions in degree $a+b$}. Similarly for $\XX$. If the actual dimension is bigger than the expected value, we say that the linear system is defective and we call {\bf defect} the difference between the expected dimension and the actual dimension. In these cases, we call {\bf algebraic defect} the difference between the actual dimension and the virtual dimension. Note that the algebraic defect might be bigger than the defect.

In case of small number of points, the dimension of linear systems of curves with multiple base points of any multiplicity is known. This story goes back to the work of G. Castelnuovo \cite{Ca91} and attracted a lot of attention in the commutative algebra and algebraic geometry community. Just to mention some of them, see \cite{Na60,Segre61,DG84,Harbourne85,Gimigliano88,Hirschowitz89,AH,AH00}. For any number of points with not too large multiplicities see \cite{CM00,CCMO03,Ya07}.

\subsection{Lemmata}
\text{The following results,} are well-known facts for the experts in the area and can be found in several papers in the literature. \textcolor{black}{We explicitly recall for convenience of the reader.}

\begin{lemma}\label{lemma: extremal}
Let $Z$ be a $0$-dimensional scheme in $\PP^2$. Then:
\begin{enumerate}
 \item if $Z$ imposes independent conditions in degree $d$, then, it is true also for any $Z'\subset Z$;
 \item if $\caL_d(Z)$ is empty, then $\caL_d(Z'')$ is empty for any $Z'' \supset Z$.
\end{enumerate}
\end{lemma}
\begin{proof}
 (1) It is enough to consider the following chain of inequalities:
 \begin{align*}
	{d+2 \choose 2} - \deg(Z') \leq \dim\caL_d(Z') & \leq \dim\caL_d(Z) + \deg(Z\setminus Z') = \\ & = {d+2 \choose 2} - \deg(Z) + (\deg(Z) - \deg(Z')).
 \end{align*}
 
 (2) If there are no curves of degree $d$ through $Z$, then, there are no curves of degree $d$ through $Z''$.
\end{proof}
  
  \begin{lemma}\label{lemma: special} 
   Let $Z$ be a scheme of fat points in $\PP^2$ in general position. If there exists a specialization $\widetilde{Z}$ of $Z$ such that $\caL_d(\widetilde{Z})$ is non-defective, then it is true also for $\caL_d(Z)$.
  \end{lemma}
  \begin{proof}
   It follows by upper semicontinuity of the Hilbert function.
  \end{proof}
  
  \begin{remark}\label{remark: extremal}
  %\textcolor{red}{--rileggi--}
 The analogous of Lemma \ref{lemma: extremal} and Lemma \ref{lemma: special} also hold if we consider schemes of fat points in $\PP^1\times\PP^1$. This tells us that, in order to prove that a scheme $\XX$ of $s$ fat points of multiplicity $m$ in $\PP^1\times \PP^1$ imposes independent conditions in bi-degree $(a,b)$ for {\it any} number of point $s$, it is enough to consider 
$$
 s_1 = \left\lfloor \frac{(a+1)(b+1)}{{m+1 \choose 2}} \right\rfloor ~~~ \text{ and }~~~
 s_2 = \left\lceil \frac{(a+1)(b+1)}{{m+1 \choose 2}} \right\rceil,
$$
and to prove that:
\begin{itemize}
\item $\caL_{(a,b)}(\XX)$ has expected dimension for $s = s_1$, i.e., equal to $(a+1)(b+1) - s_1{m+1 \choose 2}$;
\item $\caL_{(a,b)}(\XX)$ is empty for $s = s_2$.
\end{itemize}
\end{remark}
  
  \begin{notation}
   Let $Z$ be a scheme of fat points and $r = \{L = 0\}$ be a line in $\PP^2$. We denote:
   \begin{enumerate}
   \item[] ${\rm Res}_r(Z)$: the {\bf residue} of $Z$ with respect to $r$ is the scheme defined by the ideal $I_Z : (L)$;
   \item[] ${\rm Tr}_r(Z)$: the {\bf trace} of $Z$ over $r$ is the scheme defined by the ideal $I_Z + (L)$.
   \end{enumerate}
   More explicitly, if $Z = m_1P_1 + \ldots + m_sP_s$ and the points $P_1,\ldots,P_{s'}$ have support on the line $r$, we have that
   \begin{align*}
{\rm Res}_r(Z) & = (m_1-1)P_1 + \ldots + (m_{s'}-1)P_{s'} + m_{s'+1}P_{s'+1} + \ldots + m_sP_s \subset \PP^2; \\
 {\rm Tr}_r(Z) &= m_1P_1 + \ldots + m_{s'}P_{s'} \subset r.
\end{align*}
  \end{notation}
\begin{lemma}\label{lem:collinear points}{\rm \cite[Lemma 2.2]{CGG05}}
 Let $Z \subset \PP^2$ be a $0$-dimensional scheme, and let $P_1,\ldots,P_s$ be general points on a line $r$. 
\begin{enumerate}
\item If $ \dim\caL_{d}(Z+P_1+\ldots+P_{s-1}) > \dim\caL_{d-1}({\rm Res}_r(Z))$, then
$
  \dim\caL_d(Z+P_1+\ldots+P_s) = \dim\caL_d(Z) - s;$
\item 
if $ \dim\caL_{d-1}({\rm Res}_r(Z))=0$ and $\dim\caL_{d}(Z)\leq s$, then
$\dim\caL_{d}(Z+P_1+\ldots+P_{s})= 0$.
\end{enumerate}
\end{lemma}
% (2) It follows by (1) with an easy inductive idea. Assume that $\dim\caL_{d}(Z) = s \geq 1$ and we prove that for a set of general collinear points of cardinality $s'$, for any $s' \leq s$, we have $\dim\caL_{d}(Z + P_1 + \ldots + P_{s'}) = s - s'$. If $s'$, this is trivial. Assume that our claim is true for some $s' < s$, i.e., $\dim\caL_{d}(Z + P_1 + \ldots + P_{s'}) = s - s' > 0$. Then, by (1), $\dim\caL_{d}(Z + P_1 + \ldots + P_{s'} + P_{s'+1}) = s- s' -1$.

\subsection{Horace method}\label{sec:Horace}
The Horace method provides a very powerful tool to prove that a base-curve free linear system has the expected dimension by using an inductive approach. 

Let $Z = m_1P_1+\ldots+m_sP_s \subset \PP^2$ be a scheme with support in general points in $\PP^2$. By parameter count, we know that
\begin{equation}\label{ineq:dimension}
 \dim\caL_d(Z) \geq \max\left\{0, {d+2 \choose 2} - \sum_{i=1}^s {m_i+1 \choose 2}\right\}.
\end{equation}
By Lemma \ref{lemma: special}, if we find a specialization $\widetilde{Z}$ of our scheme such that $\dim\caL_d(\widetilde{Z})$ is as expected, we conclude that the same is true for $Z$. We specialize some of the points to be collinear. Assume that $P_1,\ldots,P_{s'}$ lie on the line $r = \{L = 0\}$. Then, we have {\it Castelnuovo's inequality},
\begin{equation}\label{ineq:castelnuovo}
 \dim~(I_{\widetilde{Z},\PP^2})_d \leq \dim~(I_{{\rm Res}_r(\widetilde{Z}),\PP^2})_{d-1} + \dim~(I_{{\rm Tr}_r(\widetilde{Z}),r})_d.
\end{equation}
This inequality allows us to use induction because on the right hand side we have the \textcolor{black}{dimension of} linear system $\caL_{d-1}({\rm Res}_r(\widetilde{Z}))$ of plane curves with lower degree and the dimension of the ideal of a $0$-dimensional scheme \textcolor{black}{embedded in $\PP^1$}. Thus, if we can prove that the right hand side in \eqref{ineq:castelnuovo} equals the right hand side of \eqref{ineq:dimension}, we can conclude. Unfortunately, sometimes the arithmetic does not allow this method to work for any specialization. In order to overcome this problem, J.~Alexander and A.~Hirschowitz introduced in a series of papers the so-called {\it differential Horace method} \cite{AHb,AHa,AH,AH00}. Here, we follow the exposition of \cite{GI04}.

\begin{definition}
	In the ring of formal functions $S = \CC\llbracket x,y \rrbracket$, we say that an ideal is
	{\it vertically graded with respect to $y$} if it is of the form
	$$
	I = I_0 \oplus I_1y \oplus I_2y^2 \oplus \ldots \oplus (y^m),\text{ where the } I_i
	\text{'s are ideals in }\CC\llbracket x\rrbracket.
	$$
	Let $Z$ be a $0$-dimensional scheme in $\PP^2$ with support at a point $P$ lying on line $r$. We say that $Z$ is {\it vertically graded with base $r$} if there is a regular system of parameters $(x,y)$ at $P$ such that $r$ is defined by $y = 0$ and the ideal of $Z$ is vertically graded in the localization of the coordinate ring of $\PP^2$ at the point $P$. 

For any positive integer $t$, we define the $(t+1)$-th residue and trace of $Z$ with respect to $r$ by the ideals:
\begin{align*}
\text{\it $(t+1)$-th residue:} & ~~~I_{{\rm Res}^t_r(Z)} := I_Z + (I_Z : I_r^{t+1})I_r^t; \\ 
 \text{\it $(t+1)$-th trace:} & ~~~I_{{\rm Tr}^t_r(Z)} := (I_Z : I_r^t) \otimes \mathcal{O}_r;
\end{align*}
where $\mathcal{O}_r$ denotes the structure sheaf of the line $r$. In ${\rm Res}^t_r(Z)$, we remove the $(t+1)$-th {\it slice} of $Z$; in ${\rm Tr}^t_r(Z)$, we consider only the $(t+1)$-th {\it slice} of $Z$. If $Z = Z_1 + \ldots + \textcolor{black}{Z_s}$ is a non-connected $0$-dimensional scheme of fat points, we denote by ${\rm Res}^t_r(Z) = {\rm Res}^t_r(Z_1) + \ldots + {\rm Res}^t_r(\textcolor{black}{Z_s})$ and ${\rm Tr}^t_r(Z) = {\rm Tr}^t_r(Z_1)+ \ldots + {\rm Tr}^t_r(Z_s)$.
\end{definition}
\begin{example}
 Let $Z = 3P$ be a triple point in $\PP^2$. Let $I_P = (x,y)$ and $r = \{y = 0\}$. We have that $Z$ is vertically graded with base $r$ because we can write $I_Z = I_0 \oplus I_1y \oplus I_2y^2 \oplus (y^3)$, where $I_i = (x^{3-i}) \subset \CC\llbracket x\rrbracket$. 
 
 \noindent If $t = 0$, we obtain the usual residue and trace, i.e.,
  
 \begin{minipage}{0.75\textwidth}
  \begin{align*}
   I_{{\rm Res}^0_r(Z)} &:= (x^3,x^2y,xy^2,y^3) + \left((x^3,x^2y,xy^2,y^3) : (y)\right) = \\ & = (x^2,xy,y^2)\subset \CC\llbracket x,y \rrbracket;\\
      I_{{\rm Tr}^0_r(Z)} &:= (x^3,x^2y,xy^2,y^3) \otimes \frac{\CC\llbracket x,y \rrbracket}{(y)} = (x^3) \subset \frac{\CC\llbracket x,y \rrbracket}{(y)}.
  \end{align*}
  \end{minipage}\hspace{1cm}
  \begin{minipage}{0.2\textwidth}
  \resizebox{!}{0.6\textwidth}{
  \begin{tikzpicture}[line cap=round,line join=round,x=1.0cm,y=1.0cm]
\clip(3.,-1.) rectangle (7.,3.);
\draw [domain=0.:10.] plot(\x,{(-0.-0.*\x)/1.});
\begin{scriptsize}
\draw[color=black] (-4.14,-0.15) node {$f$};
\draw [fill=black] (4.,0.) circle (4.5pt);
\draw [fill=black] (5.,0.) circle (4.5pt);
\draw [fill=black] (6.,0.) circle (4.5pt);
\draw [color=black] (4.,1.) circle (4.5pt);
\draw [color=black] (5.,1.) circle (4.5pt);
\draw [color=black] (4.,2.) circle (4.5pt);
\end{scriptsize}
\end{tikzpicture}}
  \end{minipage}
  
\noindent  If we consider other values of $t$, we consider different {\it slices} on the line $r$. If $t = 1$, we get
  
   \begin{minipage}{0.75\textwidth}
  \begin{align*}
   I_{{\rm Res}^1_r(Z)} &:= (x^3,x^2y,xy^2,y^3) + \left((x^3,x^2y,xy^2,y^3) : (y^2)\right)(y) = \\ & = (x^3,xy,y^2)\subset \CC\llbracket x,y \rrbracket;\\
      I_{{\rm Tr}^1_r(Z)} &:= \left((x^3,x^2y,xy^2,y^3):(y)\right) \otimes \frac{\CC\llbracket x,y \rrbracket}{(y)} = (x^2) \subset \frac{\CC\llbracket x,y \rrbracket}{(y)}.
  \end{align*}
  \end{minipage}\hspace{1cm}
  \begin{minipage}{0.2\textwidth}
  \resizebox{!}{0.6\textwidth}{
  \begin{tikzpicture}[line cap=round,line join=round,x=1.0cm,y=1.0cm]
\clip(3.,-1.) rectangle (7.,3.);
\draw [domain=0.:10.] plot(\x,{(-0.-0.*\x)/1.});
\begin{scriptsize}
\draw[color=black] (-4.14,-0.15) node {$f$};
\draw [fill=black] (4.,0.) circle (4.5pt);
\draw [fill=black] (5.,0.) circle (4.5pt);
\draw [color=black] (6.,1.) circle (4.5pt);
\draw [color=black] (4.,1.) circle (4.5pt);
\draw [color=black] (5.,1.) circle (4.5pt);
\draw [color=black] (4.,2.) circle (4.5pt);
\end{scriptsize}
\end{tikzpicture}}
  \end{minipage}

\noindent If $t = 2$, we get 

   \begin{minipage}{0.75\textwidth}
  \begin{align*}
   I_{{\rm Res}^2_r(Z)} &:= (x^3,x^2y,xy^2,y^3) + \left((x^3,x^2y,xy^2,y^3) : (y^3)\right)(y^2) = \\
   & = (x^3,x^2y,y^2)\subset \CC\llbracket x,y \rrbracket;\\
      I_{{\rm Tr}^2_r(Z)} &:= \left((x^3,x^2y,xy^2,y^3):(y^2)\right) \otimes \frac{\CC\llbracket x,y \rrbracket}{(y)} = (x) \subset \frac{\CC\llbracket x,y \rrbracket}{(y)}.
  \end{align*}
  \end{minipage}\hspace{1cm}
  \begin{minipage}{0.2\textwidth}
  \resizebox{!}{0.6\textwidth}{
  \begin{tikzpicture}[line cap=round,line join=round,x=1.0cm,y=1.0cm]
\clip(3.,-1.) rectangle (7.,3.);
\draw [domain=0.:10.] plot(\x,{(-0.-0.*\x)/1.});
\begin{scriptsize}
\draw[color=black] (-4.14,-0.15) node {$f$};
\draw [fill=black] (4.,0.) circle (4.5pt);
\draw [color=black] (5.,2.) circle (4.5pt);
\draw [color=black] (6.,1.) circle (4.5pt);
\draw [color=black] (4.,1.) circle (4.5pt);
\draw [color=black] (5.,1.) circle (4.5pt);
\draw [color=black] (4.,2.) circle (4.5pt);
\end{scriptsize}
\end{tikzpicture}}
  \end{minipage}
\end{example}

\begin{lemma}[{\it (Lemma d'Horace diff\'erentielle)}]\cite[Proposition 2.6]{GI04}
 Let $Z = R + S\subset\PP^2$ be a scheme of fat points where $R$ is a $0$-dimensional scheme of general fat points with support on a line $r$ and $S$ is a scheme of general fat points with support on $\PP^2$. If
 \begin{enumerate}
  \item $\dim~(I_{{\rm Res}^t_r(R) + S,\PP^2})_{d-1} = {d+1 \choose 2} - \deg({\rm Res}^t_r(R) + S)$;
  \item $\dim~(I_{{\rm Tr}^t_r(R),r})_d = d+1 - \deg({\rm Tr}^t_r(R))$;
 \end{enumerate}
 then, 
 $$
  \dim\caL_d(Z) = {d+2 \choose 2} - \deg(Z).
 $$
\end{lemma}

 \section{Hilbert function in low bi-degrees}\label{sec:3}
We are now ready to start our computations. We use the following notation. Let $\XX$ denote a scheme of $s$ fat points of multiplicity $m$ in $\PP^1 \times \PP^1$, i.e., $$\XX = mP_1 + \ldots + mP_s.$$ Let $Z$ be the scheme of fat points in $\PP^2$ constructed from $\XX$ and a bi-degree $(a,b)$, where we always assume $a\geq b \textcolor{black}{ \geq 0}$, as described in Lemma \ref{lem:multiproj-aff-proj}, i.e., $$X = aQ_1 + bQ_2 + mP_1 + \ldots + mP_s.$$
We can consider multiplicities $m\geq 2$ since the case of simple points is trivial. %\textcolor{red}{(serve???)}.

 Given two points $A,B \in \PP^2$, we use the notation $\overline{AB}$ for the line passing through them.
 
 \subsection{The case $m = \min\{a,b\}$.}\label{sec: m = b} In this section, \textcolor{black}{we start our computations by considering $m = b$}. Since the cases $b = 0$ is trivial, we may assume $b \geq 1$.

Let $a = bk + c$, with $0 \leq c \leq b-1$. Then, as we mentioned in Remark \ref{remark: extremal}, we consider
\begin{align*}
 s_1 & = \left\lfloor \frac{(a+1)(b+1)}{{b+1 \choose 2}} \right\rfloor = 
 \begin{cases}
  2k & \text{ for } c < \frac{b-2}{2}; \\
  2k+1 & \text{ for } \frac{b-2}{2} \leq c < b-1; \\
  2k+2 & \text{ for } c = b-1.
 \end{cases} \\
  s_2 & = \left\lceil \frac{(a+1)(b+1)}{{b+1 \choose 2}} \right\rceil = 
  \begin{cases}
  2k+1 & \text{ for } c < \frac{b-2}{2}; \\
  2k+2 & \text{ for } c \geq \frac{b-2}{2}. 
 \end{cases}
\end{align*}

We first prove the following lemma that will be useful for our methods. 
\begin{lemma}\label{lemma: multiple curve}
	Let $X = aQ_1 + bQ_2 + bP_1 + \ldots + bP_s \subset \PP^2$ where $a = bk+c$, with $0 \leq c \leq b-1$, and $s = 2k+1$. Then, the unique (irreducible) curve $C \in \caL_{k+1}(kQ_1+Q_2+P_1+\ldots+P_s)$ is contained in the base locus of $\caL_{a+b}(X)$ with multiplicity at least $b-c$.
\end{lemma}
\begin{proof}
	We proceed by induction on $b-c$. First, observe that, if $C'$ is a general element in $\caL_{a+b}(X)$, we have
	\begin{align*}
		\deg(C\cap C') & = ak + (s+1)b = \deg(C)\deg(C') + b-c;
	\end{align*}
	then, by B\'ezout's Theorem, $C$ is contained in $\caL_{a+b}(X)$. If $b-c = 1$, the claim is obvious. Let $b-c \geq 2$. We can remove the curve and we get $\caL_{a'+b'}(X')$, where $X' = a'Q_1 + b'Q_2 + b'P_1 + \ldots + b'P_s$, with $a' = a-k$ and $b' = b-1$. Now, $a' = b'k+c$ where $c \leq b-2 = b' -1$. By inductive hypothesis, the curve $C$ is contained with multiplicity $b'-c = b-1-c$ in $\caL_{a'+b'}(X')$ and, consequently, the claim follows. 
\end{proof}

Before the general case, we consider particular cases depending on the congruence class of $a$ modulo $b$. 

\begin{lemma}\label{prop:congruent_0_defective}{\rm [$a \equiv 0$ (mod $b$)]}
 Let $X = bkQ_1 + bQ_2 + bP_1 + \ldots + bP_s\subset \PP^2$. Then,
 $$
  \dim\caL_{b(k+1)}(X) = \max\left\{0, {b(k+1) + 2 \choose 2} - {bk+1 \choose 2} - (s+1){b+1 \choose 2}\right\}.
 $$
 except for $s = 2k + 1$, where the defect is equal to $1$.
\end{lemma}
\begin{proof}
 {\sc[Case $s = 2k+1$]} In this case, we expect the linear system $\caL_{b(k+1)}(X)$ to be empty. The conclusion follows because, by Lemma \ref{lemma: multiple curve}, the unique (irreducible) curve $C$ in the linear system $\caL_{k+1}(kQ_1+Q_2+P_1+\ldots+P_{2k+1})$ is contained with multiplicity $b$ in the base locus of $\caL_{b(k+1)}(X)$. 
 
\smallskip
\noindent {\sc [Case $s > 2k+1$]} Since $\dim\caL_{b(k+1)}(X) = 1$ for $s = 2k+1$, the linear system is empty for $s > 2k+1$. 
 
 \smallskip
 \noindent {\sc [Case $s = 2k$]} We know that
 \begin{equation*}
  \dim\caL_{b(k+1)}(X) \geq {b(k+1) + 2 \choose 2} - {bk+1 \choose 2} - (2k+1){b+1 \choose 2} = b+1.
 \end{equation*}
 Let $\widetilde{X}$ be the specialized scheme where we assume that the points $P_{1},\ldots,P_{k}$ are collinear with $Q_2$ and lie on a line $r$ (here, with an abuse of notation, we still call the specialized points by $P_i$'s). By Lemma \ref{lemma: special}, it is enough to prove the following.
 
 \smallskip
 {\it Claim: $\dim\caL_{b(k+1)}(\widetilde{X}) = b+1$.}
 \smallskip
  
  We add an extra point $A$ on the line $r$ and consider the scheme $\widetilde{X}+A$. If we prove that $$\dim\caL_{b(k+1)}(\widetilde{X}+A) = b,$$
 we are done. The line $r$ is a fixed component for $\caL_{b(k+1)}(\widetilde{X}+A)$; hence, we can remove it and
 $$
 	\dim\caL_{b(k+1)}(\widetilde{X}+A) = \dim\caL_{b(k+1)-1}({\rm Res}_r(\widetilde{X}+A)).
 $$
 Now, we proceed by induction on $b$.
 
If $b = 1$, 
$$
	\dim\caL_{k}({\rm Res}_r(\widetilde{X}+A)) = {k+2 \choose 2} - {k+1 \choose 2} - k = 1 = b.
$$
Now, let $b \geq 2$. The lines $\overline{Q_1P_i}$ are a fixed component of $\dim\caL_{b(k+1)-1}({\rm Res}_r(\widetilde{X}+A))$, for $i = k+1,\ldots , 2k$. Hence, after removing them, we conclude by induction
 \begin{equation}\label{eq:8b}
  \dim\caL_{b(k+1)-1)}({\rm Res}_r(\widetilde{X}+A)) = \dim\caL_{(b-1)(k+1)}(X') = b,
 \end{equation}
 where $X'=(b-1)k Q_1 + (b-1) Q_2 + (b-1)P_1 + \ldots + (b-1)P_{2k}$. Now, the claim is proved.
\end{proof}
\begin{lemma}\label{prop:congruent_b-1_defective}{\rm [$a \equiv b-1$ (mod $b$)]}
 Let $X = \textcolor{black}{(bk+b-1)}Q_1 + bQ_2 + bP_1 + \ldots + bP_s\subset\PP^2$. Then, 
 $$
  \dim\caL_{b(k+2)-1}(X) = \max\left\{0,{b(k+2)+1 \choose 2}-{b(k+1) \choose 2}-(s+1){b+1 \choose 2}\right\}.
 $$
\end{lemma}
\begin{proof}
 In the notations of Remark \ref{remark: extremal}, we have that $s_1 = s_2 = 2k+2$. Hence, we just need to prove that $\caL_{b(k+2)-1}(X)$ is empty in the case of $2k+2$ points. Let $\widetilde{X}$ be the specialized scheme, where the points $Q_2,P_{1},\ldots,P_{k+1}$ lie on a line $r$. By Lemma \ref{lemma: special}, it is enough to prove the following claim. 
 
 \smallskip
 	{\it Claim: $\dim\caL_{b(k+2)-1}(\widetilde{X}) = 0.$}
 	
 \smallskip 
  The line $r$ is a fixed component of the linear system $\caL_{b(k+2)-1}(X)$, hence
$$
	\caL_{b(k+2)-1}(\widetilde{X}) = \caL_{b(k+2)-2}({\rm Res}_r(\widetilde{X})).
$$
 We proceed by induction on $b$. If $b = 1$, 
 \begin{align*}
 	\dim\caL_{k+1}(\widetilde{X}) = \dim\caL_{k}({\rm Res}_r(\widetilde{X})
 	= {k+2 \choose 2} - {k+1 \choose 2} - (k+1) = 0. 
 \end{align*}
Assume $b\geq 2$. The lines $\overline{Q_1P_i}$ are contained in the base locus of $\caL_{b(k+2)-2}({\rm Res}_r(\widetilde{X}))$, for $i = k+2,\ldots, 2k+2$, and, after removing them, we conclude by induction 
 \begin{equation}\label{eq:13a}
  \dim\caL_{b(k+2)-2}(\widetilde{X}) = \dim\caL_{(b-1)(k+2)-1}(X') = 0,
 \end{equation}
 where $X '=\big((b-1)(k+1)-1\big)Q_1 + (b-1)Q_2 + (b-1)P_1 + \ldots + (b-1)P_{2k+2}$. Now, the claim is proved.
\end{proof}

\begin{proposition}\label{prop:2k+1 points}
 Let  $X = aQ_1 + bQ_2 + bP_1 + \ldots + bP_s\subset \PP^2$ where $a = bk+c$, with $0 \leq c \leq b-1$, and $s = 2k+1$. Then, 
% $$\HF_X(a+b) = {a+1 \choose 2} + (2k+2){b+1 \choose 2} - {b-c \choose 2}.$$ 
 $$ 
 	\dim\caL_{a+b}(X) = {a+b+2 \choose 2} - {a+1 \choose 2} - (2k+2){b+1 \choose 2} + {b-c \choose 2} = {c+2 \choose 2}.
 $$
 In particular, for $b-c \geq 2$, it is defective.
\end{proposition}
\begin{proof}
 By Lemma \ref{lemma: multiple curve}, the unique (irreducible) curve $C$ in the linear system $\caL_{k+1}(kQ_1+Q_2+P_1+\ldots+P_s)$ is contained with multiplicity $b-c$ in $\caL_{a+b}(X)$. Therefore, after removing it, we obtain
 $$
 	\dim\caL_{a+b}(X) = \dim\caL_{c(k+2)}(X'),
 $$
 where $X' = c(k+1)Q_1 + cQ_2 + cP_1 + \ldots + cP_s$. Then, by Lemma \ref{prop:congruent_0_defective}, we get
 $$
 	\dim\caL_{a+b}(X)  = {c(k+2)+2 \choose 2} - {c(k+1)+1 \choose 2} - (2k+2){c+1 \choose 2} = {c+2 \choose 2}.
 $$
\end{proof}
\begin{proposition}\label{prop:2k+2 points}
 Let $X = aQ_1 + bQ_2 + bP_1 + \ldots + bP_s \subset \PP^2$ where $a = bk+c$, with $0 \leq c \leq b-1$, and $s = 2k+2$. Then, $$\dim\caL_{a+b}(X) = 0.$$
\end{proposition}
\begin{proof}
We proceed by induction on $b-c$. If $b-c = 1$, then the claim follows from Lemma \ref{prop:congruent_b-1_defective}. Then, we assume $b-c \geq 2$. The curve $C \in \caL_{k+1}(kQ_1 + Q_2 + P_1 + \ldots + P_{s-1})$ is contained in the base locus of the linear system. After removing it, we obtain that
  $$
   \dim\caL_{b(k+1)+c}(X) = \dim\caL_{a'+b'}(X'),
  $$
  where $X' = a'Q_1+ b'Q_2 + b'P_1 + \ldots + b'P_{s-1} + bP_s$ where $a' = a-k$ and $b' = b-1$. Consider the subscheme $X'' = a'Q_1+ b'Q_2 + b'P_1 + \ldots + b'P_{s-1} + b'P_s$. Since $c \leq b-2 = b'-1$, we conclude by induction that the linear system $\caL_{a'+b'}(X'')$ is empty and, {\it a fortiori}, also $\caL_{a'+b'}(X')$ is empty.
\end{proof}

\begin{proposition}
 Let $X = aQ_1 + bQ_2 + bP_1 + \ldots + bP_s \subset \PP^2$ where $a = bk+c$, with $0 \leq c \leq b-1$, and $s = 2k$. Then, 
 $$
 	\dim\caL_{a+b}(X) = {a+b+2 \choose 2} - {a+1 \choose 2} - (2k+1){b+1 \choose 2} = (b+1)(c+1).
 $$
\end{proposition}
\begin{proof}
  We proceed by induction on $b-c$. If $b - c = 1$, then $a \equiv b-1 ~{\rm (mod }~b{\rm)}$ and the claim follows from Lemma \ref{prop:congruent_b-1_defective}. Let $b - c \geq 2$. We consider extra points $A_1,\ldots,A_{c+1}$ where $A_1$ is general and $A_2,\ldots,A_{c+1}$ lie on the unique curve in the linear system $\caL_{k+1}(kQ_1+Q_2+P_1+\ldots+P_s+A_1)$. Thus, by B\'ezout's Theorem, $C$ is fixed component for the linear system of curves ofdegree $a+b$ through $X + A_1+\ldots + A_{c+1}$. Hence,
\begin{equation}\label{13:b}
  \dim\caL_{a+b}(X+A_1+\ldots+A_{c+1}) = \dim\caL_{a'+b'}(X'),
\end{equation}
 where $X' = a'Q_1+ b'Q_2 + b'P_1 + \ldots + b'P_s$, with $a' = a-k$ and $b' = b-1$. Now, $a' = b'k+c$ with $c \leq b-2 = b'-1$. By induction, 
 \begin{equation}\label{13:c}
 	\dim\caL_{a'+b'}(X') = (b'+1)(c+1) = b(c+1).
 \end{equation}
 Since
\begin{equation}
 \dim\caL_{a+b}(X) \leq \dim\caL_{a+b}(X+A_1+ \ldots +A_{c+1}) + (c+1)
\end{equation}  
and by \eqref{13:b} and \eqref{13:c}, we get
$$
	\dim\caL_{a+b}(X) \leq (b+1)(c+1).
$$
Since the expected dimension is always a lower bound for the actual dimension, we conclude.
 \end{proof}

Summarizing all previous results, we obtain the following result.

\begin{theorem}\label{thm:main}
  Let $X = aQ_1 + bQ_2 + bP_1 + \ldots + bP_s \subset \PP^2$ with $a\geq b$. Then,
  $$
  	\HF_X(a+b) = \min\left\{{a+b+2 \choose 2}, {a+1 \choose 2} + (s+1){b+1 \choose 2}\right\},
  $$
  except for $s = 2k+1$ and $a = bk+c$, with $0 \leq c \leq b-2$, where 
  $$
  	\HF_X(a+b) = {a+b+2 \choose 2} - {c+2 \choose 2}.
  $$
\end{theorem}

\begin{remark}
	In the case $b = 2$, this result was already proved in \cite[Proposition 2.1]{CGG05}.
\end{remark}

\subsection{The case $m > \min\{a,b\}$.}\label{sec: m > b}
Let $X = aQ_1 + bQ_2 + mP_1 +\ldots+ mP_s \subset\PP^2$ with $a \geq b$ and $m > b$. First, note that if the linear system $\caL_{a+b}(X)$ is not empty, \textcolor{black}{then}, by B\'ezout's Theorem, we have that all the lines $\overline{Q_1P_i}$ are contained in the base locus  with multiplicity at least $m-b$. If $a - s(m-b) \geq b$, we have that
$$
	\dim\caL_{a+b}(X) = \dim\caL_{a'+b}(X'),
$$
with $X' = a'Q_1 + bQ_2 + bP_1 + \ldots + bP_s$, where $a' = a-s(m-b)$. Therefore, this case can be reduced to the case $m = b$ that we \textcolor{black}{ treated in the previous} section. Here, we \textcolor{black}{consider the case} $a - s(m-b) < b$.
\begin{proposition}\label{prop: a' < b}
	Let $X = aQ_1 + bQ_2 + mP_1 + \ldots + mP_s\subset \PP^2$ with $a \geq b$, $m > b$ and $a - s(m-b) < b$. Then, $\caL_{a+b}(X)$ is empty, \textcolor{black}{except for $s = 1$, with $0 \leq a+b-m < b$, where $\dim\caL_{a+b} = {a+b-m+2 \choose 2}$.}
\end{proposition}
\begin{proof}
	Assume that $a - s(m-b) < 0 $. \textcolor{black}{Let $s'$, $m'$ be the quotient and the reminder, respectively, of the division between $a$ and $m-b$, i.e.,} $a = s'(m-b) + m'$, with $s' < s$ and $0 \leq m' < m-b$. Then,
	$$
		\dim\caL_{a+b}(X) = \dim\caL_{b}(X'),
	$$
	where $X' = \textcolor{black}{0\cdot Q_1} + bQ_2 + bP_1 + \ldots + bP_{s'} + (m-m')P_{s'+1} + mP_{s'+2} + \ldots +mP_s$. Since $m - m' > b$, we conclude that $\caL_{b}(X')$ is empty. 

	Now, we are left with the cases $0 \leq a-s(m-b) < b$.
	
	 \textcolor{black}{If $s = 1$, we have $m \leq a + b$ and $a < m$. By B\'ezout's Theorem, we can remove all the lines $\overline{Q_1P_1}$ with multiplicity $m-b$ and $\overline{Q_2P_1}$ with multiplicity $m-b$. Then, we get
	 $$
	 	\dim\caL_{a+b}(X) = \dim\caL_{2(a+b-m)}((a+b-m)Q_1+(a+b-m)Q_2+(a+b-m)P).
	 $$
	 This is non-defective and we get $\dim\caL_{a+b}(X) = {a+b-m+2\choose 2}$.}
	
	 If $s \geq 2$, we remove the lines $\overline{Q_1P_i}$, for $i = 1,\ldots,s$, \textcolor{black}{with multiplicity $m-b$} and we get
	$$
		\caL_{a+b}(X) = \caL_{a'+b}(X''),
	$$
	where $X'' = a'Q_1 + bQ_2 + bP_1 + \ldots + bP_s$ and $a' = a-s(m-b) < b$. \textcolor{black}{Since $s \geq 2$}, it is enough to show that $\caL_{a'+b}(X''')$, where $X''' = a'Q_1 + bQ_2 + bP_1 + bP_2$, is empty. Assume $\dim\caL_{a'+b}(X''') \neq 0$. Since the lines $\overline{Q_2P_1}, \overline{Q_2P_2}$ and $\overline{P_1P_2}$ are in the base locus of $\caL_{a'+b}(X''')$ with multiplicity at least $b-a'$, we need to have $a'+b \geq 3(b-a')$. Let 
	$$b-a' = \begin{cases} 2t & \text{ if } b-a' \text{ is even}; \\ 2t+1 & \text{ if } b-a' \text{ is odd.}\end{cases}$$
	We remove the lines $\overline{Q_2P_1}, \overline{Q_2P_2}$ and $\overline{P_1P_2}$ with multiplicity $t$ and we get
	$$
		\dim\caL_{a'+b}(X''') = \dim\caL_{a'+b-3t}(\overline{X}),
	$$
	where $\overline{X} = a'Q_1 + b'Q_2 + b'P_1 + b'P_2$ where $b' = b-2t$. Note that, since $a'+b \geq 3(b-a')$, then $a'+b-3t \geq 0$. If $b - a' = 2t$, we have $a' = b'$ and, since $t\geq 1$, 
	$
		\dim\caL_{2b'-t}(\overline{X}) = 0.
	$
	If $b - a' = 2t+1$, we have $a' = b'-1$ and 
	$
		\dim\caL_{2b'-t-1}(\overline{X}) = 0.
	$
	This concludes the proof.
\end{proof}

\subsection{Bi-graded Hilbert function in extremal bi-degrees.} 
By Lemma \ref{lem:multiproj-aff-proj}, we can translate our previous computations to get the expressions for the bi-graded Hilbert function of schemes of general fat points in $\PP^1\times\PP^1$ in {\it extremal} bi-degrees.
  \begin{theorem}\label{thm:main_any_m}
  Let $a \geq b$ and assume $m \geq b$. Let $\XX = mP_1 + \ldots + mP_s \subset \PP^1 \times\PP^1$. Then, 
 $$
  \HF_\XX(a,b) = \min\left\{(a+1)(b+1),~ 
  	s{m+1 \choose 2} - s{m-b \choose 2}\right\}
 $$ 
 except if $s = 2k+1$ and $a = bk+c+s(m-b)$, with $c = 0,\ldots,b-2$, where 
 $$
 	\HF_\XX(a,b) = (a+1)(b+1) - {c+2 \choose 2}.
 $$
 \end{theorem}
\begin{proof}
 By using the multiprojective-affine-projective method, we need to look at the linear system $\caL_{a+b}(X)$ with $X = aQ_1 + bQ_2 + mP_1 +\ldots + mP_s$. 
 
\textcolor{black}{Let $a-s(m-b) < b$. If $s \geq 2$, by Proposition \ref{prop: a' < b}, $\dim\caL_{a+b}(X) = 0$. Since in this case the inequality}
\begin{equation}\label{ineq:main_thm}
 (a+1)(b+1) \leq s{m+1 \choose 2} - s{m-b \choose 2},
\end{equation}
\textcolor{black}{holds, we conclude. If $s = 1$ and $a-s(m-b) <  0$, again by Proposition \ref{prop: a' < b}, $\dim\caL_{a+b}(X) = 0$. Since the inequality \eqref{ineq:main_thm} is still true, we conclude. If $0 \leq a +b - m < b$, by Proposition \ref{prop: a' < b}, we have that}
$$
	\dim\caL_{a+b}(X) = {a+b-m+2 \choose 2} = {c+2 \choose 2}.
$$

Now, assume $a - s(m-b) \geq b$. As explained in Section \ref{sec: m > b}, we may reduce to the case where $m = b$, namely, 
 $$
 	\dim\caL_{a+b}(X) = \dim\caL_{a'+b}(X'),
 $$
 where $X' = a'Q_1 + bQ_2 + bP_1 +\ldots + bP_s$, with $a' = a-s(m-b) \geq b$. Hence, by Theorem \ref{thm:main}, the dimension of $\caL_{a'+b}(X')$ is the maximum between $0$ and
 \begin{align*}
 	{a-s(m-b)+b+2 \choose 2} & - {a-s(m-b) + 1 \choose 2} - (s+1){b+1 \choose 2} = \\
 	& = {a+b+2 \choose 2} - {a+1 \choose 2} - {b+1 \choose 2} - s{m+1 \choose 2} + s{m-b \choose 2} = \\
 	& = (a+1)(b+1) - s{m+1 \choose 2} + s{m-b \choose 2},
 \end{align*}
 except for $s = 2k+1$ and $a - s(m-b) = bk+c$, with $0 \leq c \leq b-1$, where the dimension is ${c+2 \choose 2}$. Since by Lemma \ref{lem:multiproj-aff-proj} we have $\dim\caL_{(a,b)}(\XX) = \dim\caL_{a+b}(X)$, we conclude.
\end{proof} 

\begin{remark}
 In \cite[Corollary 3.4]{GVT}, the authors proved that, for any row and column, in the notation of the latter theorem, the bi-graded Hilbert function of $\XX$ is constant for $\max(a,b) \geq sm$. Theorem \ref{thm:main_any_m} improves this result and tells us that the $b$-th column, for $b \leq m$, becomes constant for $a \geq sm - \left\lceil \frac{sb}{2} \right\rceil$.  
\end{remark}
 \begin{corollary}\label{cor:main_HF}
  Let $\XX$ be a scheme of $s$ fat points of multiplicity $m$ and in general position in $\PP^1\times \PP^1$. Assume that $a \geq b$ and set $k = \left\lfloor \frac{s}{2} \right\rfloor$.  Then, for $b \in \{m-1,m\}$, the Hilbert function of $\XX$ in the $b$-th column is constant for $ a \geq b(k+1) + s(m-b) - 1$ and equal to the degree of $\XX$, i.e., $\HF_\XX(a,b) = s{m+1 \choose 2}$.
 \end{corollary}
 \begin{proof}
  It follows from \ref{thm:main_any_m} by computing the Hilbert function in bi-degrees $((m-1)(k+1)+s-1,m-1)$ and $(m(k+1)-1,m)$ and checking that it is equal to the degree of $\XX$. 
 \end{proof}
 
\noindent A nice property of $0$-dimensional schemes is that their Hilbert function is eventually constant to the degree of the scheme. Corollary \ref{cor:main_HF} gives us lower bounds on the bi-degrees for which the Hilbert function gets constant. Hence, we are left with a limited squared unknown region. This area can be restricted by using \cite[Remark 5.4]{SV06}.

\begin{example}\label{example:HF}
We give an explicit example to describe the situation after Theorem \ref{thm:main_any_m}. Here, we look at the Hilbert function of $5$ random points of multiplicity $5$ in $\PP^1\times\PP^1$. The computation has been done with the algebra software \textit{Macaulay2} {\rm\cite{M2}}. In the table we underline the defective cases. The shaded region indicates the area that we are not yet able to compute with our result. Note that the lower right corner that has been removed from the squared region is due to \cite[Remark 5.4]{SV06} which computes the Hilbert function in any bi-degrees $(i,j)$ such that $i+j \geq \max\{sm-1, 2m-2\} = \max\{24,9\} = 24$.

\smallskip
\begin{center}
{\scriptsize
\begin{tikzpicture}
        \matrix [matrix of math nodes,left delimiter=(,right delimiter=), inner sep=3pt] (m)
        {
1&
      2&
      3&
      4&
      5&
      6&
      7&
      8&
      9&
      10&
      11&
      12&
      13&
      14&
      15&
      16&
      17&
      18&
      19&
      20&
      21&
      22&
      23&
      24&
      25& 
      \textcolor{red}{\underline{25}}\\
      2&
      4&
      6&
      8&
      10&
      12&
      14&
      16&
      18&
      20&
      22&
      24&
      26&
      28&
      30&
      32&
      34&
      36&
      38&
      40&
      42&
      44&
      \textcolor{red}{\underline{45}}&
      \textcolor{red}{\underline{45}}&
      \textcolor{red}{\underline{45}}&
      \textcolor{red}{\underline{45}}\\
      3&
      6&
      9&
      12&
      15&
      18&
      21&
      24&
      27&
      30&
      33&
      36&
      39&
      42&
      45&
      48&
      51&
      54&
      57&
      \textcolor{red}{\underline{59}}&
      \textcolor{red}{\underline{60}}&
      \textcolor{red}{\underline{60}}&
      \textcolor{red}{\underline{60}}&
      \textcolor{red}{\underline{60}}&
      \textcolor{red}{\underline{60}}&
      \textcolor{red}{\underline{60}}&\\
      4&
      8&
      12&
      16&
      20&
      24&
      28&
      32&
      36&
      40&
      44&
      48&
      52&
      56&
      60&
      64&
      \textcolor{red}{\underline{67}}&
      \textcolor{red}{\underline{69}}&
      \textcolor{red}{\underline{70}}&
      \textcolor{red}{\underline{70}}&
      \textcolor{red}{\underline{70}}&
      \textcolor{red}{\underline{70}}&
      \textcolor{red}{\underline{70}}&
      \textcolor{red}{\underline{70}}&
      \textcolor{red}{\underline{70}}&
      \textcolor{red}{\underline{70}}\\
      5&
      10&
      15&
      20&
      25&
      30&
      35&
      40&
      45&
      50&
      55&
      60&
      65&
      \textcolor{red}{\underline{ 69}}&
      \textcolor{red}{\underline{ 72}}&
      \textcolor{red}{\underline{ 74}}&
      75&
      75&
      75&
      75&
      75&
      75&
      75&
      75&
      75&
      75\\
      6&
      12&
      18&
      24&
      30&
      36&
      42&
      48&
      54&
      60&
      \textcolor{red}{\underline{ 65}}&
      \textcolor{red}{\underline{ 69}}&
      \textcolor{red}{\underline{ 72}}&
      \textcolor{red}{\underline{ 74}}&
      75&
      75&
      75&
      75&
      75&
      75&
      75&
      75&
      75&
      75&
      75&
      75\\
      7&
      14&
      21&
      28&
      35&
      42&
      \textcolor{black}{49}&
      \textcolor{black}{56}&
      \textcolor{black}{63}&
      \textcolor{red}{\underline{69}}&
      \textcolor{red}{\underline{ 72}}&
      \textcolor{red}{\underline{ 74}}&
      \textcolor{red}{\underline{ 75}}&
      \textcolor{red}{\underline{ 75}}&
      75&
      75&
      75&
      75&
      75&
      75&
      75&
      75&
      75&
      75&
      75&
      75\\
      8&
      16&
      24&
      32&
      40&
      48&
      \textcolor{black}{56}&
      \textcolor{black}{64}&
      \textcolor{red}{\underline{71}}&
      \textcolor{red}{\underline{74}}&
      75&
      75&
      75&
      75&
      75&
      75&
      75&
      75&
      75&
      75&
      75&
      75&
      75&
      75&
      75&
      75\\
      9&
      18&
      27&
      36&
      45&
      54&
      \textcolor{black}{63}&
      \textcolor{red}{\underline{71}}&
      \textcolor{black}{75}&
      \textcolor{black}{75}&
      \textcolor{black}{75}&
      \textcolor{black}{75}&
      \textcolor{black}{75}&
      \textcolor{black}{75}&
      75&
      75&
      75&
      75&
      75&
      75&
      75&
      75&
      75&
      75&
      75&
      75\\
      10&
      20&
      30&
      40&
      50&
      60&
      \textcolor{red}{\underline{69}}&
      \textcolor{red}{\underline{74}}&
      \textcolor{black}{75}&
      \textcolor{black}{75}&
      \textcolor{black}{75}&
      \textcolor{black}{75}&
      \textcolor{black}{75}&
      \textcolor{black}{75}&
      75&
      75&
      75&
      75&
      75&
      75&
      75&
      75&
      75&
      75&
      75&
      75\\
      11&
      22&
      33&
      44&
      55&
      65&
      \textcolor{red}{\underline{72}}&
      \textcolor{black}{75}&
      \textcolor{black}{75}&
      \textcolor{black}{75}&
      \textcolor{black}{75}&
      \textcolor{black}{75}&
      \textcolor{black}{75}&
      \textcolor{black}{75}&
      75&
      75&
      75&
      75&
      75&
      75&
      75&
      75&
      75&
      75&
      75&
      75\\
      12&
      24&
      36&
      48&
      60&
      \textcolor{red}{\underline{69}}&
      \textcolor{red}{\underline{74}}&
      \textcolor{black}{75}&
      \textcolor{black}{75}&
      \textcolor{black}{75}&
      \textcolor{black}{75}&
      \textcolor{black}{75}&
      \textcolor{black}{75}&
      \textcolor{black}{75}&
      75&
      75&
      75&
      75&
      75&
      75&
      75&
      75&
      75&
      75&
      75&
      75\\
      13&
      26&
      39&
      52&
      65&
      \textcolor{red}{\underline{72}}&
      \textcolor{black}{75}&
      \textcolor{black}{75}&
      \textcolor{black}{75}&
      \textcolor{black}{75}&
      \textcolor{black}{75}&
      \textcolor{black}{75}&
      \textcolor{black}{75}&
      \textcolor{black}{75}&
      75&
      75&
      75&
      75&
      75&
      75&
      75&
      75&
      75&
      75&
      75&
      75\\
      14&
      28&
      42&
      56&
      \textcolor{red}{\underline{69}}&
      \textcolor{red}{\underline{74}}&
      \textcolor{black}{75}&
      \textcolor{black}{75}&
      \textcolor{black}{75}&
      \textcolor{black}{75}&
      \textcolor{black}{75}&
      \textcolor{black}{75}&
      \textcolor{black}{75}&
      \textcolor{black}{75}&
      75&
      75&
      75&
      75&
      75&
      75&
      75&
      75&
      75&
      75&
      75&
      75\\
      15&
      30&
      45&
      60&
      \textcolor{red}{\underline{72}}&
      75&
      75&
      75&
      75&
      75&
      75&
      75&
      75&
      75&
      75&
      75&
      75&
      75&
      75&
      75&
      75&
      75&
      75&
      75&
      75&
      75\\
      16&
      32&
      48&
      64&
      \textcolor{red}{\underline{74}}&
      75&
      75&
      75&
      75&
      75&
      75&
      75&
      75&
      75&
      75&
      75&
      75&
      75&
      75&
      75&
      75&
      75&
      75&
      75&
      75&
      75\\
      17&
      34&
      51&
      \textcolor{red}{\underline{67}}&
      75&
      75&
      75&
      75&
      75&
      75&
      75&
      75&
      75&
      75&
      75&
      75&
      75&
      75&
      75&
      75&
      75&
      75&
      75&
      75&
      75&
      75\\
      18&
      36&
      54&
      \textcolor{red}{\underline{69}}&
      75&
      75&
      75&
      75&
      75&
      75&
      75&
      75&
      75&
      75&
      75&
      75&
      75&
      75&
      75&
      75&
      75&
      75&
      75&
      75&
      75&
      75\\
      19&
      38&
      57&
      \textcolor{red}{\underline{70}}&
      75&
      75&
      75&
      75&
      75&
      75&
      75&
      75&
      75&
      75&
      75&
      75&
      75&
      75&
      75&
      75&
      75&
      75&
      75&
      75&
      75&
      75\\
        };  
        \draw[fill opacity=0.2, color=black, fill = gray] (m-7-7.north west) -- (m-7-14.north east) -- (m-11-14.south east) -- (m-11-13.south east) -- (m-12-13.south east)--(m-12-12.south east)--(m-13-12.south east)--(m-13-11.south east)--(m-14-12.south west) -- (m-14-7.south west) -- (m-7-7.north west); 
        \end{tikzpicture}
        }
        \end{center}
      \end{example}
     
 \section{Triple points}\label{sec:4}
\noindent In this section, we complete Theorem \ref{thm:main_any_m} in the case of triple points in $\PP^1 \times \PP^1$. By Lemma \ref{lem:multiproj-aff-proj}, we want to compute all the dimensions of the linear systems $\caL_{a+b}(X)$ where 
$$X = aQ_1 + bQ_2 + 3P_1 + \ldots + 3P_s \subset \PP^2.$$
Accordingly with Remark \ref{remark: extremal}, we first focus on the two extremal cases
$s=s_1$ and $s=s_2$
where 
$$ s_1 = \left\lfloor \frac{(a+1)\cdot (b+1)}{6} \right\rfloor,\ \ \ \   s_2 = \left\lceil \frac{(a+1)\cdot (b+1)}{6} \right\rceil.
$$
Considering the results of the previous section, we only have to consider the cases with $a,b \geq 4$. Due to technical reasons in our general argument, we prefer to separately consider the cases $(a,b) = (4,4), (5,4)$.

\begin{lemma}[Case $(4,4)$]\label{(4,4)}
 Let $X = 4Q_1 + 4Q_2 + 3P_1 + \ldots + 3P_s\subset \PP^2.$ Then, for any $s$, the Hilbert function of $X$ in degree $8$ is as expected, i.e.,
 $$
 	\HF_X(8) = \min\left\{45,20+6s\right\}.
 $$ 
\end{lemma}
\begin{proof}
We need to prove that $X$ imposes independent conditions in degree $8$ for $s = s_1 = 4$. Since, for $s = 4$, $\dim\caL_8(X) = 1$, we would have also that $\caL_8(X)$ is empty for $s > 4$. Hence, by Lemma \ref{lemma: extremal}, we conclude.

Let $s = 4$. By Lemma \ref{lemma: special}, it is enough to prove that, for a generic point $A$, $\caL_8(X+A)$ is empty. Consider the unique (irreducible) cubic in the linear system $\caL_3(2Q_1+Q_2+P_1+\ldots+P_4+A)$. By B\'ezout's Theorem, this cubic is a fixed component of the linear system and
$$
 \dim\caL_8(X+A) = \dim\caL_5(X'),
 $$
 with $X' = 2Q_1+3Q_2+2P_1+\ldots+2P_4$. By Theorem \ref{thm:main}, we have that $\caL_5(X')$ is empty and we conclude.
\end{proof}
\begin{lemma}[Case $(5,4)$]\label{(5,4)}
 Let $X = 5Q_1 + 4Q_2 + 3P_1 + \ldots + 3P_s\subset \PP^2.$ Then, the Hilbert function of $X$ in degree $9$ is as expected, i.e., 
 $$
 	\HF_X(9) = \min\{55,25+6s\},
 $$
 except for $s = 5$ where $\HF_X(9) = 54$, instead of $55$.
\end{lemma}
\begin{proof}
If $s = s_1 = s_2 = 5$, we consider the unique (irreducible) cubic $C$ in $\caL_3(2Q_1+Q_2+P_1+\ldots+P_5)$. By B\'ezout's Theorem, the curve is contained in the base locus of $\caL_9(X)$. Then, by \cite[Proposition 2.1]{CGG05}, we have
 $$
  \dim\caL_9(X) = \dim\caL_6(3Q_1+3Q_2+2P_1+\ldots+2P_5) = 1.
 $$
 Since the linear system has dimension $1$ for $s = 5$, then $\caL_9(X)$ is empty for $s > 5$. 
 
 Consider now $s = 4$. We need to show that $\dim\caL_9(X)=6$. Let $A_1,A_2$ be two points such that $A_1$ is general and $A_2$ lies on the unique (irreducible) cubic in $\caL_3(2Q_1+Q_2+P_1+\ldots+P_4+A_1)$. By Lemma \ref{lemma: special}, it is enough to show that $\dim\caL_9(X+A_1+A_2) = 4$. By B\'ezout's Theorem, the cubic is in the base locus. Then,
 $$
  \dim\caL_9(X+A_1+A_2) = \dim\caL_6(X'),
 $$
 where $X' = 3Q_1 + 3Q_2 + 2P_1 + \ldots + 2P_4$. By \cite[Proposition 2.1]{CGG05}, we conclude.
\end{proof}

\noindent From now on, let $a\geq b\geq 4$ and $a+b \geq 10$. Our computations will be structured as follows.

\smallskip

\hfill
\begin{minipage}{0.95\textwidth}
\begin{enumerate}
\item[\it Step 1:] 
Let $r$ be a general line. We specialize the scheme $X$ to a scheme $\widetilde{X}$ having some of the triple points with support lying generically on $r$, but, with suitable degrees in such a way that, by differential Horace method, the line $r$ and the line $\overline{Q_1Q_2}$ become fixed components and can be removed.
Let $T := {\rm Res}_{\overline{Q_1Q_2}}({\rm Res}_{r}(\widetilde{X}))$ be the residual scheme.

\item[\it Step 2:]
 If necessary, we specialize another point $P_i$ on the line $r$ in such a way that the lines $r$ and $\overline{Q_1Q_2}$ are again fixed components and we can remove them. Let $\widetilde{T}$ such a specialization and consider $W := {\rm Res}_{\overline{Q_1Q_2}}({\rm Res}_{r}(\widetilde{T}))$ the residual scheme.

\item[\it Step 3:] 
The scheme $W$ has a some of the points which are in general position over the line $r$. Then, we use induction on $b$ and Lemma \ref{lem:collinear points} to conclude.
\end{enumerate}
\end{minipage}

\smallskip

Our procedure will depend on the congruence class of $a+b$ modulo $5$. The reason of this dependency will be clear during the proof and it will be caused by our particular approach.

\begin{notation}
 Recalling the constructions of Section \ref{sec:Horace}, we denote by $D^{(i)}_r(P)$ the $0$-dimensional scheme defined by $(x^i,y)$, where $\{x,y\}$ are a regular system of parameters at $P$ such that $r$ is the line $y = 0$. More in general, let $D_r^{(i_1,\ldots,i_m)}(P)$ be the $0$-dimensional scheme with support at $P$ and such that ${\rm Tr}_r^j\left(D^{(i_0,\ldots,i_m)}_r(P)\right) = D^{(i_j)}_r(P)$, for any $j = 1,\ldots,m$. We denote by $\caS_r(3P)$ the {\it slice} of the triple point that we want to consider on the line $r$.
\end{notation}

\begin{lemma}\label{lemma: main triple 1}
  Let $X = aQ_1 + bQ_2 + 3P_1 + \ldots + 3P_s \subset\PP^2$ with $s \geq s_1$, $a\geq b\geq 4$ and $a+b \geq 10$. Let $a + b = 5h + c$, with $0 \leq c \leq 4$, and let
 $$x = 
 \begin{cases}
  h+1 & \text{ for } c = 0; \\
  h+2 & \text{ for } c = 1,2,3,4; \\
 \end{cases}
 ~~~~~
 y = \begin{cases}
   h-1 & \text{ for } c = 0; \\
   h-2 & \text{ for } c = 1,2,3,4. \\
 \end{cases}.
$$
Let $\widetilde{X}$ be a specialization of $X$ having $P_1,\ldots,P_{x+y}$ lying generically on a line $r$ and, in cases $c=2,3,4$, having also $P_{x+y+1}$ lying on the line $r$, with the following degrees
$$
  \deg(\caS_r(3P_i))  = \begin{cases} 3 & \text{ for } i =1,\ldots,x; \\
  2 & \text{ for } i = x+1,\ldots,x+y; \end{cases} ~~~~~
  \deg(\caS_r(3P_{x+y+1})) = \begin{cases}
   1 & \text{ for } c = 2; \\
   2 & \text{ for } c = 3; \\
   3 & \text{ for } c = 4;
  \end{cases}
$$
 Then:
 \begin{enumerate}
 
\item $x+y+1 \leq s_1 = \left\lfloor \frac{(a+1)\cdot (b+1)}{6} \right\rfloor;$
\item $
  \dim\caL_{a+b}(\widetilde{X}) = \dim\caL_{a+b-2}\left({\rm Res}_{\overline{Q_1Q_2}}\left({\rm Res}_r(\widetilde{X})\right)\right).
 $
 \item ${\rm Res}_{\overline{Q_1Q_2}}\left({\rm Res}_r(\widetilde{X})\right)= (a-1)Q_1 + (b-1)Q_2 ~ + $
 \end{enumerate}
 $$
+\begin{cases}
  2P_1 + \ldots + 2P_{h+1} +  D_r^{(3,1)}(P_{h+2})+ \ldots + D_r^{(3,1)}(P_{2h}) +  3P_{2h+1}\ldots + 3P_{s}, \hfill \text{ for } c = 0; \\
  2P_1 + \ldots + 2P_{h+2} + D_r^{(3,1)}(P_{h+3})+ \ldots + D_r^{(3,1)}(P_{2h})   +  3P_{2h+1}\ldots + 3P_{s},
  \hfill \text{ for } c = 1; \\
 2P_1 + \ldots + 2P_{h+2} + D_r^{(3,1)}(P_{h+3})+ \ldots + D_r^{(3,1)}(P_{2h}) + D_r^{(3,2)}(P_{2h+1})  +  3P_{2h+2}\ldots + 3P_{s}, \\
 \hfill \text{ for } c = 2; \\
  2P_1 + \ldots + 2P_{h+2} +  D_r^{(3,1)}(P_{h+3})+ \ldots + D_r^{(3,1)}(P_{2h}) + D_r^{(3,1)}(P_{2h+1}) +  3P_{2h+2}\ldots + 3P_{s},\\
  \hfill \text{ for } c = 3;  \\
 2P_1 + \ldots + 2P_{h+2} +  D_r^{(3,1)}(P_{h+3})+ \ldots + D_r^{(3,1)}(P_{2h}) + 2(P_{2h+1}) +  3P_{2h+2}\ldots + 3P_{s}. \\
 \hfill \text{ for } c = 4.
 \end{cases}
 $$
\end{lemma}
\begin{proof}
(i) It is enough to show that $\frac{(a+1)(b+1)}{6} \geq x+y+1 = 2h+1$. Since 
\begin{align*}\frac{(a+1)\cdot (b+1)}{6}- ( 2h +1 )
%& =
%\frac{(a+1)\cdot (b+1)}{6}-  \frac{2(a+b-c)}{5} -1 \\
& \geq
\frac{a(5b-7)-7b-25 }{30} \geq\frac{(b-4)(5b+6)-1 }{30},
\end{align*}
then, in case  $b> 4$, we are done. For  $b=4$, we have $a \geq 6$, hence
$a(5b-7)-7b-25 = 13a -53>0$.

 (ii) First, by B\'ezout's Theorem, we prove that $r$ is a fixed component for $\caL_{a+b}(\widetilde{X})$. Now,  for a general $C \in \caL_{a+b}(\widetilde{X})$, we have 
 $$\deg (r \cap C) =
  3x+2y+\deg(\caS_r(3P_{x+y+1})) = 5h + c + 1 = a+b+1.
%  \begin{cases}
%   3(h+1) + 2(h-1) = 5h+1 & \text{ for } c = 0; \\
%   3(h+2) + 2(h-2) = 5h+2 & \text{ for } c = 1; \\
%   3(h+2) + 2(h-2) + 1 = 5h+3 & \text{ for } c = 2; \\
%   3(h+2) + 2(h-2) + 2 = 5h+4 & \text{ for } c = 3; \\
%   3(h+2) + 2(h-2) + 3 = 5h+5 & \text{ for } c = 4.
%  \end{cases} = 5h+c+1
 $$
 So, we remove the line $r$. Since in ${\rm Res}_r(\widetilde{X})$ the points $Q_1$ and $Q_2$ have still multiplicity $a$ and $b$, respectively, the line $\overline{Q_1Q_2}$ is a fixed component for $\caL_{a+b-1}\left({\rm Res}_r(\widetilde{X})\right)$, and we are done.
 
 (iii) Easily follows from the definition of $\widetilde{X}$.
\end{proof}
%%%%%%%%%%%%%%%%%%%%%%%%%

\begin{lemma}\label{lemma: main triple 2}
 Notation as in Lemma \ref{lemma: main triple 1}. Denote by $T := {\rm Res}_{\overline{Q_1Q_2}}\left({\rm Res}_r(\widetilde{X})\right)$. Let $\widetilde{T}$ be a specialization of $T$ such that, in cases $c=1$, $\widetilde{T}$ has also $P_{x+y+1}$ lying generically on the line $r$, and, in case $c=3,4$, $\widetilde{T}$ has also $P_{x+y+2}$ lying on the line $r$, with the following degrees
$$
  \deg(\caS_r(3P_i))  = \begin{cases} 2 & \text{ for } i =1,\ldots,x; \\
  3 & \text{ for } i = x+1,\ldots,x+y; \end{cases} 
  $$
  $$
  \deg(\caS_r(3P_{x+y+1})) = \begin{cases}
    2 & \text{ for } c = 1; \\
   3 & \text{ for } c = 2; \\
   3 & \text{ for } c = 3; \\
   2 & \text{ for } c = 4.
  \end{cases} \quad \quad
  \deg(\caS_r(3P_{x+y+2})) = \begin{cases}
 1& \text{ for } c = 3; \\
   3 & \text{ for } c = 4.
  \end{cases}
$$
 Denote by $ W:= {\rm Res}_{\overline{Q_1Q_2}}\left({\rm Res}_r(\widetilde{T})\right)$.
 Then:
 \begin{enumerate}

\item[(i)]  for $c=3,4,$
$$x+y+2 \leq s_1 = \left\lfloor \frac{(a+1)\cdot (b+1)}{6} \right\rfloor;$$
 \item[(ii)]  $
  \dim\caL_{a+b-2}(\widetilde{T}) = \dim\caL_{a+b-4}\left({\rm Res}_{\overline{Q_1Q_2}}\left({\rm Res}_r(\widetilde{T})\right)\right);$
  \item[(iii)]$W= (a-2)Q_1 + (b-2)Q_2 + P_1 + \ldots + P_{2h}
+ $
$$+\begin{cases}
   3P_{2h+1}\ldots + 3P_{s}  & \text{ for } c = 0; \\
  D_r^{(3,1)}(P_{2h+1})+  3P_{2h+2}\ldots + 3P_{s}  & \text{ for } c = 1; \\
  D_r^{(2)}(P_{2h+1})+  3P_{2h+2}\ldots + 3P_{s} & \text{ for } c = 2; \\
   P_{2h+1} + D_r^{(3,2)}(P_{2h+2})+  3P_{2h+3}\ldots + 3P_{s} & \text{ for } c = 3;  \\
 P_{2h+1} +  2P_{2h+2}+  3P_{2h+3}\ldots + 3P_{s} & \text{ for } c = 4;
 \end{cases}$$
 and
$$
{\rm Res}_r(W) =(a-2)Q_1 + (b-2)Q_2
+ \begin{cases}
   3P_{2h+1}\ldots + 3P_{s}  & \text{ for } c = 0; \\
  P_{2h+1}+  3P_{2h+2}\ldots + 3P_{s}  & \text{ for } c = 1; \\
   3P_{2h+2}\ldots + 3P_{s} & \text{ for } c = 2; \\
  + D_r^{(2)}(P_{2h+2})+  3P_{2h+3}\ldots + 3P_{s} & \text{ for } c = 3;  \\
 P_{2h+2}+  3P_{2h+3}\ldots + 3P_{s} & \text{ for } c = 4.
 \end{cases}
$$

\end{enumerate}

\end{lemma}
\begin{proof}
(i) It is enough to prove that $\frac{(a+1)(b+1)}{6} \geq x+y+2 = 2h + 2$. Since $b\geq 4$, 
\begin{align*}
\frac{(a+1)\cdot (b+1)}{6}- ( 2h +2 ) 
%&=
%\frac{(a+1)\cdot (b+1)}{6}-  \frac{2(a+b-c)}{5} -2 \\
& \geq
\frac{a(5b-7)-7b+36-55 }{30}\geq \frac{5b^2-14b-19 }{30}.
\end{align*}

(ii)
 We prove that $r$ is a fixed component for $\caL_{a+b-2}(\widetilde{T})$. Now,  for a general $C \in \caL_{a+b-2}(\widetilde{T})$, we have 
 \begin{align*}\deg (r \cap C) &=
  2x+3y+\deg(\caS_r(3P_{x+y+1})) +\deg(\caS_r(3P_{x+y+2})) = 5h + c - 1 = a+b-1.
%  \\ &= 
% \begin{cases}
%   2(h+1) + 3(h-1) = 5h-1 & \text{ for } c = 0; \\
%   2(h+2) + 3(h-2) +2= 5h & \text{ for } c = 1; \\
%    2(h+2) + 3(h-2)+ 3 = 5h+1 & \text{ for } c = 2; \\
%   2(h+2) + 3(h-2) + 3+1 = 5h+2 & \text{ for } c = 3; \\
%    2(h+2) + 3(h-2)+2+3  = 5h+3 & \text{ for } c = 4.
%  \end{cases} = 5h + c - 1.
 \end{align*}
So, by B\'ezout's Theorem, we may remove the line $r$.
 Since in ${\rm Res}_r(\widetilde{T})$ the points $Q_1$ and $Q_2$  have multiplicity $a-1$ and $b-1$, respectively, we have that the line $\overline{Q_1Q_2}$ is a fixed component of $\caL_{a+b-3}\left({\rm Res}_r(\widetilde{T})\right)$.

(iii) Easily follows from Lemma \ref{lemma: main triple 1}(iii) and the definition of $\widetilde{T}$.
\end{proof}

We are ready to complete our computations in the case of triple points. The final result is the following.

\begin{theorem} \label{thm:main_triple}
Let $X = aQ_1 + bQ_2 + 3P_1 + \ldots + 3P_s \subset \PP^2$, with $a\geq b\geq1$. Then, 
$$
	\HF_X(a+b) = \min\left\{{a+b+2 \choose 2} , {a+1 \choose 2} + {b+1 \choose 2} + 6s\right\},
$$
as expected, except for
\begin{enumerate}
\item $b = 1$ and $s < \frac{2}{5}(a+1)$, where $\HF_X(a+1) = {a+1 \choose 2} + 1 + 5s$;
\item  $s$ odd, say $s = 2k+1$, and 
\begin{enumerate}
	\item $(a,b) = (4k+1,2)$, where $\HF_X(4k+3) = {a+b+2 \choose 2} - 1$;
	\item $(a,b) = (3k,3)$, where $\HF_X(3k+3) = {a+b+2 \choose 2} - 1$;
	\item $(a,b) = (3k+1,3)$, where $\HF_X(3k+4) = {a+1 \choose 2} + {b+1 \choose 2} + 6s- 1$;
\end{enumerate}
	\item $s = 5$ and $(a,b) = (5,4)$, where $\HF_X(9) = 54$, instead of $55$.
\end{enumerate}
 \end{theorem}
 
\begin{proof}
The cases $b \leq 3$ and  $(a,b)=(4,4),(5,4)$ follow from Theorem \ref{thm:main}, Theorem \ref{thm:main_any_m}, Lemma \ref{(4,4)} and Lemma  \ref{(5,4)}, respectively.  Let $b\geq 4$, $a+b \geq 10$ and set $a + b = 5h + c$, with $0 \leq c \leq 4.$
By Remark \ref{remark: extremal}, we need to show:
\begin{itemize}
\item for $s = s_1$, $\caL_{a+b}(X)$ has dimension as expected, $\dim\caL_{a+b}(X) = (a+1)(b+1)-6s_1$;
\item for $s = s_2$, $\caL_{a+b}(X)$ is empty.
\end{itemize} 
By the previous lemmas, we reduce to the linear system $\caL_{a+b-4}(W)$, where $W := {\rm Res}_{\overline{Q_1Q_2}}\left({\rm Res}_r(\widetilde{T})\right)$ is as in Lemma \ref{lemma: main triple 1} and Lemma \ref{lemma: main triple 2}. If we prove that, for $s=s_1$, $\dim\caL_{a+b-4}(W) = (a+1)(b+1)-6s_1$ and, for $s=s_2$, $\caL_{a+b-4}(W) $  is empty, then, by the semicontinuity of the Hilbert function, we are done.

\medskip
\noindent {\sc [Case $s=s_1$]}
Recall that, by Lemma \ref{lemma: main triple 1} and Lemma \ref{lemma: main triple 2}:
 $$W= (a-2)Q_1 + (b-2)Q_2+ P_1 + \ldots + P_{2h}
+ \begin{cases}
   3P_{2h+1}\ldots + 3P_{s},  \hfill \text{ for } c = 0; \\
  D_r^{(3,1)}(P_{2h+1})+  3P_{2h+2}\ldots + 3P_{s}, \hfill \text{ for } c = 1; \\
  D_r^{(2)}(P_{2h+1})+  3P_{2h+2}\ldots + 3P_{s}, \hfill \text{ for } c = 2; \\
   P_{2h+1} + D_r^{(3,2)}(P_{2h+2})+  3P_{2h+3}\ldots + 3P_{s}, \\
   \hfill \text{ for } c = 3;  \\
 P_{2h+1} +  2P_{2h+2}+  3P_{2h+3}\ldots + 3P_{s}, \hfill \text{ for } c = 4.
 \end{cases}
$$
The expected dimension of $\caL_{a+b-4}(W) $ is
\begin{align*}{\it exp}.\dim \caL_{a+b-4}(W) & = 
(a-1) (b-1)-2h  -6(s_1-2h) +2c= \\
%& = (a+1)(b+1) -2a - 2b - 2h - 6s_1 + 12h + 2c = \\
& = (a+1)(b+1) - 6 s_1 - 2(5h+c) + 10 h + 2c = (a+1)(b+1) - 6 s_1.
\end{align*}
Thus, we need to prove that, for $s=s_1$,  $W$ imposes  independent conditions to the curves of degree $a+b-4$. We use Lemma \ref{lem:collinear points}(i). Let
 $$W_1 = \begin{cases} W-\{P_1,\ldots,P_{2h}\} & \text{ for } c = 0,1,2; \\
W-\{P_1,\ldots,P_{2h+1}\} & \text{ for } c = 3,4,
 \end{cases}
$$
that is,
 $$W_1= (a-2)Q_1 + (b-2)Q_2+
 \begin{cases}
   3P_{2h+1}\ldots + 3P_{s_1}  & \text{ for } c = 0; \\
  D_r^{(3,1)}(P_{2h+1})+  3P_{2h+2}\ldots + 3P_{s_1}  & \text{ for } c = 1; \\
  D_r^{(2)}(P_{2h+1})+  3P_{2h+2}\ldots + 3P_{s_1} & \text{ for } c = 2; \\
   D_r^{(3,2)}(P_{2h+2})+  3P_{2h+3}\ldots + 3P_{s_1} & \text{ for } c = 3;  \\
 2P_{2h+2}+  3P_{2h+3}\ldots + 3P_{s_1} & \text{ for } c = 4.
 \end{cases}
$$

\smallskip
\noindent {\it Claim 1.} {\it For $s=s_1$, 
\begin{itemize}
\item $ \dim\caL_{a+b-4}(W_1+P_1+\cdots +P_{2h-1} ) > \dim\caL_{a+b-5}({\rm Res}_r(W_1)), \quad \text{ for } c = 0,1,2; $
  \item $\dim\caL_{a+b-4}(W_1+P_1+\cdots +P_{2h} ) > \dim\caL_{a+b-5}({\rm Res}_r(W_1)), \quad \text{ for } c = 3,4.$
   \end{itemize}
   }
   
   \smallskip
 
\noindent {\it Proof of Claim 1.}
 By parameter count, we know that the left hand sides are always greater or equal than $(a+1)(b+1)-6s_1+1$ which is strictly positive, by definition of $s_1$. Since the line $\overline{Q_1Q_2}$ is a fixed component for $\caL_{a+b-5}({\rm Res}_r(W_1))$, we also have 
$$\dim\caL_{a+b-5}({\rm Res}_r(W_1))=
\dim\caL_{a+b-6}({\rm Res}_{\overline{Q_1Q_2}}({\rm Res}_r(W_1))),$$
where

$${\rm Res}_{\overline{Q_1Q_2}}({\rm Res}_r(W_1))
 =(a-3)Q_1 + (b-3)Q_2
+ \begin{cases}
   3P_{2h+1}\ldots + 3P_{s},  & \text{ for } c = 0; \\
  P_{2h+1}+  3P_{2h+2}\ldots + 3P_{s},  & \text{ for } c = 1; \\
   3P_{2h+2}\ldots + 3P_{s}, & \text{ for } c = 2; \\
  D_r^{(2)}(P_{2h+2})+  3P_{2h+3}\ldots + 3P_{s}, & \text{ for } c = 3;  \\
 P_{2h+2}+  3P_{2h+3}\ldots + 3P_{s}, & \text{ for } c = 4.
 \end{cases}
$$

	Now, we want to prove our claim by induction, but we need to be careful because we might fall in one of the defective cases we have considered above.

\smallskip
\noindent (a) {\it Non-defective case}. If we do not fall in one of the defective cases, we have that, by induction and by observing that general simple points, in the case $c = 1,4$, and a general $2$-jet, in the case $c = 3$, impose independent conditions, we have
	$$
		\dim \caL_{a+b-6}({\rm Res}_{\overline{Q_1Q_2}}({\rm Res}_r(W_1))) = 
		\begin{cases}
			\max\left\{0, (a-2)(b-2) - 6(s_1-2h) \right\}, \quad  \hfill \text{ for } c = 0; \\
			\max\left\{0, (a-2)(b-2) - 6(s_1-2h) + 5\right\},  \hfill \text{ for } c = 1; \\
			\max\left\{0, (a-2)(b-2) - 6(s_1-2h) + 6\right\},  \hfill  \text{ for } c = 2; \\
			\max\left\{0, (a-2)(b-2) - 6(s_1-2h) + 10\right\},  \hfill  \text{ for } c = 3; \\
			\max\left\{0, (a-2)(b-2) - 6(s_1-2h) + 11\right\},  \hfill  \text{ for } c = 4.
		\end{cases}
	$$
	If $\caL_{a+b-6}({\rm Res}_{\overline{Q_1Q_2}}({\rm Res}_r(W_1)))$ is empty, the claim is trivial. Otherwise, for $c = 0,1,2$,
	\begin{align}\label{eq:claim1_1}
		\dim\caL_{a+b-4}&(W_1+P_1+ \cdots +P_{2h-1}) - \dim \caL_{a+b-6}({\rm Res}_{\overline{Q_1Q_2}}({\rm Res}_r(W_1))) \nonumber \\
		& \geq (a+1)(b+1)-6s_1+1 - (a-2)(b-2) + 6(s_1-2h) - 
		\begin{cases}
			0 & \text{ for } c = 0;\\
			5 & \text{ for } c = 1;\\
			6 & \text{ for } c = 2;\\
		\end{cases} \nonumber \\
		& = 3(a+b) -12h - 2 - \begin{cases}
			0 & \text{ for } c = 0;\\
			5 & \text{ for } c = 1;\\
			6 & \text{ for } c = 2;\\
		\end{cases} \geq 
		\begin{cases}
			4 & \text{ for } c = 0;\\
			2 & \text{ for } c = 1;\\
			4 & \text{ for } c = 2;\\
		\end{cases}
	\end{align}
	similarly, for $c = 3,4$, we obtain
	\begin{equation}\label{eq:claim1_2}\dim\caL_{a+b-4}(W_1+P_1+ \cdots +P_{2h}) - \dim \caL_{a+b-6}({\rm Res}_{\overline{Q_1Q_2}}({\rm Res}_r(W_1))) 
	\geq 		
	\begin{cases}
			3, \hfill ~\text{for } c = 3;\\
			5, \hfill ~\text{for } c = 4.\\
	\end{cases}
	\end{equation}
	In particular, we obtain that the Claim 1 holds under the assumption (a).	

\smallskip
\noindent (b) {\it $b-3 = 1$}. In this case, we know that we have defect; in particular, we have
	$$
		\dim \caL_{a+b-6}({\rm Res}_{\overline{Q_1Q_2}}({\rm Res}_r(W_1))) = 
		\begin{cases}
			\max\left\{0, (a-2)(b-2) - 5(s_1-2h) \right\}, \hfill \text{ for } c = 0; \\
			\max\left\{0, (a-2)(b-2) - 5(s_1-2h) + 4\right\}, \hfill \text{ for } c = 1; \\
			\max\left\{0, (a-2)(b-2) - 5(s_1-2h) + 5\right\},  \hfill \text{ for } c = 2; \\
			\max\left\{0, (a-2)(b-2) - 5(s_1-2h) + 8\right\},  \hfill \text{ for } c = 3; \\
			\max\left\{0, (a-2)(b-2) - 5(s_1-2h) + 9\right\},  \hfill \text{ for } c = 4.
		\end{cases}
	$$
	If $\caL_{a+b-6}({\rm Res}_{\overline{Q_1Q_2}}({\rm Res}_r(W_1)))$ is empty, the claim is trivial. Otherwise, for $c = 0,1,2$,
	\begin{align*}
		\dim&\caL_{a+b-4}(W_1+P_1+ \cdots +P_{2h-1}) - \dim \caL_{a+b-6}({\rm Res}_{\overline{Q_1Q_2}}({\rm Res}_r(W_1))) \\
		& \geq (a+1)(b+1)-6s_1+1 - (a-2)(b-2) + 5(s_1-2h) - 
		\begin{cases}
			0, \hfill \text{ for } c = 0;\\
			4, \hfill \text{ for } c = 1;\\
			5, \hfill \text{ for } c = 2;\\
		\end{cases} \\
		& = 3(a+4) - 10h - 2 - s_1 - \begin{cases}
			0 & \text{ for } c = 0;\\
			5 & \text{ for } c = 1;\\
			6 & \text{ for } c = 2;\\
		\end{cases} \\
		& = 5h - \frac{5(a+1)}{6} - 
		\begin{cases}
			2 & \text{ for } c = 0;\\
			3 & \text{ for } c = 1;\\
			1 & \text{ for } c = 2;\\
		\end{cases} \geq \frac{5h - 8}{6} > 0;
	\end{align*}
	similarly, for $c = 3,4$, we obtain
	\begin{align*}
	\dim\caL_{a+b-4}(W_1+P_1+ &\cdots +P_{2h}) - \dim \caL_{a+b-6}({\rm Res}_{\overline{Q_1Q_2}}({\rm Res}_r(W_1)))  \\
	& \geq 		
	5h - \frac{5(a+1)}{6} - 
		\begin{cases}
			1 & \text{ for } c = 3;\\
			-1 & \text{ for } c = 4;\\
		\end{cases} \geq \frac{5h - 6}{6} > 0.
	\end{align*}
	In particular, we obtain that the Claim 1 holds under the assumption (b).
	
\smallskip
\noindent (c) {\it $b-3 \neq 1$ and algebraic defect {\rm (for the definition, see Section \ref{sec:2})} equal to $1$.} In these cases, since the algebraic defect is equal to $1$, we may adapt the computations to obtain \eqref{eq:claim1_1} and \eqref{eq:claim1_2}. In particular, we obtain
 \begin{align*}
		\dim\caL_{a+b-4}(W_1+P_1+ \cdots +P_{2h-1}) - \dim & \caL_{a+b-6}({\rm Res}_{\overline{Q_1Q_2}}({\rm Res}_r(W_1))) \geq 
		\begin{cases}
			3, \hfill \text{ for } c = 0;\\
			1, \hfill \text{ for } c = 1;\\
			3, \hfill \text{ for } c = 2;\\
			2, \hfill \text{ for } c = 3;\\
			4, \hfill \text{ for } c = 4.\\
		\end{cases}
 \end{align*}
	Hence, Claim 1 holds also in this case.
	
	\smallskip
	\noindent (d) {\it $b-3 \neq 1$ and algebraic defect equal to $3$}. In this case, we cannot adapt the previous computations. This defective case would appear only for $(a-3,b-3) = (3k,3)$, for some positive integer $k$, and if $(a+1)(b+1) - 6s_1 + 1 = 1$. Therefore, only if $7(3k+4) = 6s_1$. This is a contradiction, cause the left hand side is clearly not divisible by $3$.

	\smallskip
	Therefore, Claim 1 is proved.

\bigskip
\noindent {\it Claim 2.} {\it For $s=s_1$, 
$W_1$ gives independent conditions to the curves of degree $a+b-4$.}

\medskip
\noindent {\it Proof of Claim 2.}
  Let
$$\overline{W}_1= 
(a-2)Q_1 + (b-2)Q_2+
 \begin{cases}
   3P_{2h+1}\ldots + 3P_{s}  & \text{ for } c = 0,1,2; \\
3P_{2h+2}+  3P_{2h+3}\ldots + 3P_{s} & \text{ for } c = 3,4.
 \end{cases}
$$
Since $\overline{W}_1 \supset W_1$, by Lemma \ref{lemma: special}, it is enough to check that the claim holds for $\overline{W}_1$.
  
  Again, we need to be careful because we might fall in a defective case. We consider them separately. 
	
	\smallskip
	\noindent (a) {\it $(a-2,b-2) = (3k+1,3)$, for some positive integer $k$.} In this case, we have that $s_1 = 3k+4$ and $3k+8 = 5h+c$. Therefore, if $c = 0,1,2$,
		$$
			s_1 - 2h < 2k+1 \Longleftrightarrow k > \begin{cases}
				1 & \text{ for } c = 0; \\
				3 & \text{ for } c = 1; \\
				5 & \text{ for } c = 2; \\
			\end{cases}
		$$
		and, if $ c = 3,4$,
		$$
			s_1 - 2h - 1 < 2k+1 \Longleftrightarrow k > \begin{cases}
				0 & \text{ for } c = 3; \\
				2 & \text{ for } c = 4. \\
			\end{cases}
		$$
		Therefore, for large values of $k$, we do not risk to fall in the defective cases. Hence, we are left with the cases $(a,b) = (6,5),(9,5),(12,5)$ where it can be easily checked by specialization that $W_1$ impose independent conditions.
		
		\smallskip
		\noindent (b) $(a-2,b-2) \neq (3k+1,3)$. In this case, observe that, if $c = 0,1,2$,
		$$
			(a-1)(b-1) - 6(s_1-2h) = (a+1)(b+1) - 6s_1 +2h - 2c \geq 2h - 2c \geq 0;
		$$
		and, if $c = 3,4$,
		\begin{align*}
			(a-1)(b-1) - 6(s_1-2h-1) & = (a+1)(b+1) - 6s_1 +2h - 2c + 6 \\ & \geq 2h - 2c + 6 \geq 2.
		\end{align*}
		Therefore, since in the defective cases, we have that the dimension of the linear system is equal to $1$, then we are left to only check the case where $c = 2$, $h = 2$ and $(a+1)(b+1) - 6s_1 = 0$. Since it has to be $a + b = 12$ and the defective cases $(a-2,b-2) = (3k,3),(5,4)$ do not satisfy this condition, a fortiori, we have that that $b - 2 = 2$ and, consequently, $(a,b) = (8,4)$ with $s_1 = 7$. However, in this case, $(a+1)(b+1) - 6s_1 \neq 0$ and we obtain a contradiction.
 
  \smallskip
	Therefore, Claim 2 is proved.

  \bigskip
Now by Lemma \ref{lem:collinear points}(i), with $W_1=Z$, and Claim 1   it follows that, 
for $s=s_1$, 
$$
\dim\caL_{a+b-4}(W) =
\begin{cases} \dim\caL_{a+b-4}(W_1)-2h & \text{ for } c = 0,1,2; \\
\dim\caL_{a+b-4}(W_1)-2h-1 & \text{ for } c = 3,4.
 \end{cases}
$$
By Claim 2 and easy computation, it follows that $\dim\caL_{a+b-4}(W) = (a+1)(b+1)-6s_1$.

Hence, {\sc Case} $s = s_1$ is proved.

\bigskip
\noindent {\sc [Case $s=s_2$]} We need to prove that, for $s=s_2$, the linear system $\caL_{a+b-4}(W) $ is empty. 

If $s_2=s_1$ the conclusion follows from the previous case. So, assume that $s_2>s_1$.
We use Lemma \ref{lem:collinear points}(ii). Let
$$W_2 = \begin{cases} W-\{P_1,\ldots,P_{2h}\} & \text{ for } c = 0,1,2; \\
W-\{P_1,\ldots,P_{2h+1}\} & \text{ for } c = 3,4;
 \end{cases}
$$
that is,
 $$W_2= (a-2)Q_1 + (b-2)Q_2+
 \begin{cases}
   3P_{2h+1}\ldots + 3P_{s_2}  & \text{ for } c = 0; \\
  D_r^{(3,1)}(P_{2h+1})+  3P_{2h+2}\ldots + 3P_{s_2}  & \text{ for } c = 1; \\
  D_r^{(2)}(P_{2h+1})+  3P_{2h+2}\ldots + 3P_{s_2} & \text{ for } c = 2; \\
   D_r^{(3,2)}(P_{2h+2})+  3P_{2h+3}\ldots + 3P_{s_2} & \text{ for } c = 3;  \\
 2P_{2h+2}+  3P_{2h+3}\ldots + 3P_{s_2} & \text{ for } c = 4.
 \end{cases}
$$
 
\bigskip
\noindent {\it Claim 3.} {\it For $s = s_2$, the linear system $\caL_{a+b-5}({\rm Res}_r(W_2))$ is empty.} 
 
\medskip
\noindent {\it Proof of Claim 3.}
As in the proof of Claim 1,  since the line $\overline{Q_1Q_2}$ is a fixed component for $\caL_{a+b-5}({\rm Res}_r(W_2))$, we have 
$$\dim\caL_{a+b-5}({\rm Res}_r(W_2))=
\dim\caL_{a+b-6}({\rm Res}_{\overline{Q_1Q_2}}({\rm Res}_r(W_2))),$$
where
$${\rm Res}_{\overline{Q_1Q_2}}({\rm Res}_r(W_2))
 =(a-3)Q_1 + (b-3)Q_2
+ \begin{cases}
   3P_{2h+1}\ldots + 3P_{s_2}  & \text{ for } c = 0; \\
  P_{2h+1}+  3P_{2h+2}\ldots + 3P_{s_2}  & \text{ for } c = 1; \\
   3P_{2h+2}\ldots + 3P_{s_2} & \text{ for } c = 2; \\
  D_r^{(2)}(P_{2h+2})+  3P_{2h+3}\ldots + 3P_{s_2} & \text{ for } c = 3;  \\
 P_{2h+2}+  3P_{2h+3}\ldots + 3P_{s_2} & \text{ for } c = 4.
 \end{cases}
$$
Note that in ${\rm Res}_{\overline{Q_1Q_2}}({\rm Res}_r(W_2))$ there are $s'$ triple points, where
$$s' = 
\begin{cases}
   s_2- 2h  & \text{ for } c = 0; \\
s_2- 2h-1 & \text{ for } c = 1,2; \\
   s_2- 2h-2 & \text{ for } c = 3,4.
 \end{cases}
$$

Again, we need to be careful to distinguish when we fall in the defective cases.

\smallskip
\noindent (a) {\it $(a-3,b-3)$ is not a defective case.} In this case, we have (recall that $a + b = 5h+c$)
	\begin{align*}
		(a-2)(b-2) - 6s' & = (a+1)(b+1) - 3(a+b) + 3 - 6s' = \\
		& = (a+1)(b+1) - 6s_2 -3h - 3c + 3 +
		\begin{cases}
			0 & \text{ for } c = 0; \\
			6 & \text{ for } c = 1,2; \\
			12 & \text{ for } c = 3,4; \\
		\end{cases} \\
		& \leq -3h + 
		\begin{cases}
			3 & \text{ for } c = 0; \\ 
			6 & \text{ for } c = 1,2; \\
			6 & \text{ for } c = 3,4. \\
		\end{cases}
	\end{align*}
	Since $h \geq 2$, we have $(a-2)(b-2) - 6s' \leq 0$ and Claim 3 holds under assumption (a).

\smallskip
\noindent (b) {\it $b-3 = 1$}. In this case, we know that triple points given $5$ condition instead of $6$. Since $a + b = 5h+c$, we have $s_2 = \left\lceil \frac{5(5h+c-3)}{6} \right\rceil$. Therefore,
	\begin{align*}
		2(a-2) - 5s' & = -5\left\lceil \frac{5(5h+c-3)}{6} \right\rceil + 10h + 2c - 12 + 
		\begin{cases}
			10h & \text{ for } c = 0; \\
			10h+5 & \text{ for } c = 1,2; \\
			10h+10 & \text{ for } c = 3,4; \\
		\end{cases} \\
		& = \begin{cases}
			-5\left\lceil \frac{h+3}{6}\right\rceil + 3  & \text{ for } c = 0; \\
			-5\left\lceil \frac{h+2}{6}\right\rceil + 5  & \text{ for } c = 1; \\
			-5\left\lceil \frac{h+1}{6}\right\rceil + 2  & \text{ for } c = 2; \\			
			-5\left\lceil \frac{h}{6}\right\rceil + 2  & \text{ for } c = 3; \\			
			-5\left\lceil \frac{h+5}{6}\right\rceil + 6  & \text{ for } c = 4. \\			
		\end{cases}
	\end{align*}
	Since $h \geq 2$, we have $(a-2)(b-2) - 5s' \leq 0$ and Claim 3 holds under assumption (b).
	
	\smallskip
	\noindent (c) {\it $b - 3 = 2$}. In this case, we have that $s_2 = s_1$ and it follows from {\sc Case} $s = s_1$.
	
	\smallskip
	\noindent (d) {\it $(a-3,b-3) = (3k,3)$}. If $k = 1$, we have $(a,b) = (6,6)$, $s_2 = 9$ and $s' = 4 > 2k+1$. Hence, in this case Claim 3 holds. Assume now $k \geq 2$. Then, $a+b \geq 15$ and $h \geq 3$. Moreover, since $a = 3k+3$ and $ a+ b = 5h + c$, we obtain $k = \frac{5h+c-9}{3}$. Therefore,
	\begin{align*}
		s' - (2k+1) & = \left\lceil \frac{7(5h+c-5)}{6} \right\rceil -2h - \frac{10h+2c-15}{3} + \begin{cases}
		0 & \text{ for } c = 0;\\
		-1 & \text{ for } c = 1,2;\\
		-2 & \text{ for } c = 3,4;\\
\end{cases}
\\ & \geq  \frac {3h+3c-5}6 + \begin{cases}
0 & \text{ for } c = 0; \\
-1  & \text{ for } c = 1,2; \\
 -2 & \text{ for } c = 3,4;
 \end{cases} =  \begin{cases}
\frac{3h-5}{6}  & \text{ for } c = 0,2,4; \\
\frac{3h-8}{6}  & \text{ for } c = 1,3.
 \end{cases}
	\end{align*}
	 Hence, since $h \geq 3$, we have that $s' > 2k+1$ and Claim 3 holds under assumption (d).
	 
	 \smallskip
	 \noindent (e) {\it $b-3 = 3$ and $a - 3 = 3k+1$.} Since $a = 3k+4$ and $a+b = 5h+c$, we get $k = \frac{5h+c-10}{3}$. Therefore,
	 \begin{align*}
	 	s'-(2k+1) & = \left \lceil  \frac {35h +7c -35} 6 \right \rceil -2h -  \frac {10h+2c-17}3 +  \begin{cases}
0 & \text{ for } c = 0; \\
-1  & \text{ for } c = 1,2; \\
 -2 & \text{ for } c = 3,4;
 \end{cases} \\
 & \geq  \frac {3h+3c-1}6 + \begin{cases}
0 & \text{ for } c = 0; \\
-1  & \text{ for } c = 1,2; \\
 -2 & \text{ for } c = 3,4;
 \end{cases} \ \ =  \begin{cases}
\frac{3h-1}{6}  & \text{ for } c = 0,2,4; \\
\frac{3h-4}{6}  & \text{ for } c = 1,3.
 \end{cases}
	 \end{align*}
	 Hence, since $h \geq 2$, we have that $s' > 2k+1$ and Claim 3 holds under assumption (e).
	 
	 \smallskip
	 \noindent (f) {\it $(a-3,b-3) = (5,4)$}. In this case, we have $h = 3$, $c = 0$ and $s_2 = 12$. Hence, $s' = s_2 - 2h = 6$, so we do not fall in the defective case and Claim 3 holds under assumption (f).

\smallskip
Hence, Claim 3 is completely proved.

\bigskip
\noindent {\it Claim 4.} {\it For $s = s_2$,
 $$\dim\caL_{a+b-4}(W_2)\leq\begin{cases} 2h  & \text{ for } c = 0,1,2; \\
2h+1 & \text{ for } c = 3,4.
 \end{cases}
$$}

\noindent {\it Proof of Claim 4.}
\textcolor{black}{
Since $a + b = 5h + c$, the expected dimension of  $\caL_{a+b-4}(W_2)$ is 
\begin{align*}
\textit{exp}.\dim\caL_{a+b-4}(W_2) &= \begin{cases}
\max \{0, (a-1)(b-1) - 6(s_2 -2h) +  0\}  & \text{ for } c = 0; \\
\max \{0, (a-1)(b-1) - 6(s_2 -2h) +  2 \}  & \text{ for } c = 1; \\
\max \{0, (a-1)(b-1) - 6(s_2 -2h) +  4 \}  & \text{ for } c = 2; \\
\max \{0, (a-1)(b-1) - 6(s_2 -2h) +  7 \}  & \text{ for } c = 3; \\
\max \{0, (a-1)(b-1) - 6(s_2 -2h) +  9 \} & \text{ for } c = 4;
 \end{cases}\\
 & = 
\begin{cases}
\max \{0, (a+1)(b+1) - 6s_2 +  
2h  \}  & \text{ for } c = 0,1,2; \\
\max \{0, (a+1)(b+1) - 6s_2 + 2h+1  \} & \text{ for } c = 3,4.
 \end{cases}
\end{align*}
Now, if $h > 2$ or if $h = 2$ and $c = 3,4$, we have that $(a+1)(b+1) - 6s_2 \geq -5$. Therefore, 
$$
\textit{exp}.\dim\caL_{a+b-4}(W_2)  = (a+1)(b+1) - 6s_2 +  \begin{cases}
2h  & \text{ for } c = 0,1,2; \\
2h+1 & \text{ for } c = 3,4.
 \end{cases}
 $$
 If $h = 2$ and $c = 0,1,2$, since we assume $s_2 > s_1$, we are left only with the following cases: $(a,b) = (6,4), (7,4), (6,6), (8,4)$. Among these cases, only for $(a,b) = (4,4)$ we have that 
 $$
 	(a+1)(b+1) - 6s_2 + 4 = 49 - 54 + 4 = -1 < 0;
 $$
 but, since $(a-2,b-2) = (4,4)$ is a non-defective case, we have $\dim\caL_{a+b-4}(W_2) = 0$. Therefore, excluding this case, we have
 $$
 	\dim\caL_{a+b-4}(W_2) \leq {\it exp}.\dim\caL_{a+b-4}(W_2) \leq \begin{cases}
2h  & \text{ for } c = 0,1,2; \\
2h+1 & \text{ for } c = 3,4;
 \end{cases}
 $$
 because, since $s_2 > s_1$, we have $(a+1)(b+1) - 6s_2 \leq -1$.  Hence, Claim 4 is proved.
}

\medskip
Now, by  Claim 3, Claim 4 and Lemma \ref{lem:collinear points}(ii), with $W_2=Z$, it follows that also {\sc Case} $s = s_2$ is proved.
\end{proof}

Therefore, as a direct corollary of Theorem \ref{thm:main_triple}, we obtain the following formulas for the complete bi-graded Hilbert function for schemes of triple points.

\begin{theorem}\label{thm:main_triple_multi}
	Let $\XX = 3P_1 + \ldots + 3P_s \subset \PP^1 \times \PP^1$. Then,
	$$
		\HF_\XX(a,b) = \min\left\{(a+1)(b+1), 6s\right\},
	$$
	except for
	\begin{enumerate}
	\item $b=1$ and $s<\frac{2}{5}(a+1)$, where $\HF_{\XX}(a,1) = 5s$;
\item  $s$ odd, say $s = 2k+1$, and \\
	 	$\quad (a,b) = (4k+1,2)$, where $\HF_\XX(4k+1,2) = (a+1)(b+1) - 1$; \\
   		$\quad (a,b) = (3k,3)$, where $\HF_\XX(3k,3) = (a+1)(b+1) - 1$;\\
	 	$\quad (a,b) = (3k+1,3)$, where $\HF_\XX(3k+1,3) = 6s- 1$; \\
	\item $s = 5$ and $(a,b) = (5,4)$, where $\HF_X(5,4) = 29$.\\
\end{enumerate}
\end{theorem}
\begin{proof}
	It directly follows from Lemma \ref{lem:multiproj-aff-proj} and Theorem \ref{thm:main_triple}.
\end{proof}
\appendix
\section{{\it Macaulay2} code}\label{sec: appendix}
In this appendix, we implement our results with the algebra software {\it Macaulay2} \cite{M2}. With the standard tools of the software, we would need to first construct the ideal of fat points by using random coordinates and then we would compute the Hilbert function with the implemented command {\tt hilbertFunction}. These computations, since they involve Gr\"obner basis, might not even finish in reasonable time. Here is a possible code to try this.

{\small
\begin{verbatim}
	-- INPUT: s = number of points;
	--        m = multiplicity;
	--        a,b = bi-degree;
	S = QQ[x_0,x_1,y_0,y_1, Degrees => {{1,0},{1,0},{0,1},{0,1}}]
	I = intersect for i from 1 to s list 
         (ideal(random(QQ)*x_0 + random(QQ)*x_1, 
                        random(QQ)*y_0+random(QQ)*y_1))^m;
	hilbertFunction({a,b},I)
\end{verbatim}
}
\medskip

Our main results Theorem \ref{thm:main_any_m} and Theorem \ref{thm:main_triple} allow us to give a numerical function which computes the Hilbert function of schemes of general fat points in $\PP^1 \times \PP^1$ very quickly, even for large numbers, where the usual functions cannot be efficient due to the Gr\"obner basis computation. Here is a possible implementation of this in the language of the algebra software {\it Macaulay2} \cite{M2}.

\medskip
{\small
\begin{verbatim}
	-- Function which returns 
	--    the binomial coefficient (m choose k) if m is greater or equal 
	--    than k and 0 otherwise;
	Bin = method();
	Bin (ZZ,ZZ) := (m,k) -> (if m >= k then return binomial(m,k) else return 0)
	
	-- INPUT: s = number of points;
	--        m = multiplicity;
	--        a,b = bi-degree
	-- OUTPUT: (if m >= b or m <= 3) Hilbert function in bi-degree (a,b) of  
	--          a scheme of s general fat points of multiplicity m
	multiFatPoints = method()
	multiFatPoints (ZZ,ZZ,ZZ,ZZ) := (m,s,a,b) -> (
	      if (m < min(a,b) and m > 3) then (
	         return "ERROR: multiplicity has to be m >= min(a,b) or m <= 3");
	      A := max(a,b); B := min(a,b);
	      if (m >= B) then (
	         if s % 2 == 1 then (
	            k := s // 2;
	            if (0 <= A-B*k-s*(m-B) and A-B*k-s*(m-B) <= B-2) then (
	               c := A - B*k - s*(m-B);
	               return ((A+1)*(B+1) - Bin(c+2,2))
	            ) else (
	               return min((A+1)*(B+1) , s*Bin(m+1,2) - s*Bin(m-B,2)))
	       ) else (
	             return min((A+1)*(B+1) , s*Bin(m+1,2) - s*Bin(m-B,2))) 
	 ) else (
	       if ( s == 5 and A == 5 and B == 4) then (
	             return 54
	       ) else (
	             return  min((A+1)*(B+1) , s*Bin(m+1,2) - s*Bin(m-B,2)))
          )
      )
\end{verbatim}
}

\section{Other defective cases}\label{appendix: defective}

\noindent We give an infinite family of defective cases for any multiplicity that is not covered from our previous computations.

\begin{proposition}
	Let $X = aQ_1 + bQ_2 + mP_1 + \ldots +mP_s$, where $a = (2m-1)(m-2),~b = m+1$, $m \geq 2$, and $s = 4m - 7$. Then, we have that 
	$\caL_{a+b}(X)$ is defective with defect $1$.
\end{proposition}
\begin{proof}The expected dimension is
 $$
  \textit{exp}.\dim\caL_{a+b}(X) = (a+1)(b+1) - s{m+1 \choose 2} = \frac{(m-3)(m-4)}{2}.
 $$
  Now, consider the unique curve $C$ of degree $2m-3$ passing simply through all the points $Q_2,P_1,\ldots,P_s$ and with multiplicity $2m-4$ at $Q_1$. Then, for a general $C'\in\caL_{a+b}(X)$, we have
 \begin{align*}
  \deg(C \cap C') & = (2m-1)(m-2)(2m-4) + (m+1) + (4m-7)m, \\
  \deg(C)\deg(C') & = ((2m-1)(m-2)+(m+1))(2m-3). 
 \end{align*}
 Hence, we get that the curve $C$ is contained in the base locus of $\caL_{a+b}(X)$ because
 \begin{align*}
  \deg(C \cap C') - \deg(C)\deg(C') & = (4m-7)m-(m+1)(2m-4)-(2m-1)(m-2) = 2.
 \end{align*}
Then, we can remove it and we obtain
$$
	\dim\caL_{a+b}(X) = \dim\caL_{a'+b'}(X'),
$$ 
 where $X' = a'Q_1 + b'Q_2 + (m-1)P_1 + \ldots + (m-1)P_s$, with $a' = a - (2m-4) = 3m-6 + (m-3)(2m-4)$ and $b' = m$. The expected dimension is
 $$
 	(a'+1)(b'+1) - s{m \choose 2} = \frac{(m-3)(m-4)}{2} + 1.
 $$
 Therefore, it is enough to prove the following claim.
 
 \medskip
\noindent {\it Claim.} Let $m \geq 3$. Consider $X_i = a_iQ_1 + b_iQ_2 + m_iP_1 +\ldots + m_iP_s$, where $a_i = 3m-6 + i(2m-4),~b_i = 3+i,~m_i = 2+i$ and $s = 4m-7$, for any $0 \leq i \leq m-3$. Then, $\caL_{a_i+b_i}(X_i)$ has dimension as expected. 
 
 \medskip
 In particular, this concludes the proof because we have that $X'$ coincides with $X_{m-3}$.
 
\medskip
\noindent {\it Proof of Claim.}
	We proceed by induction on $i$. If $i = 0$, we conclude by \cite[Proposition 2.1]{CGG05}. Let $0 < i \leq m-3$. Consider again the curve $C$ and let $\widetilde{X}_i = X_i + A_1 +\ldots + A_{m-3-i}$, where the points $A_i$'s are general on the curve $C$. The expected dimension of $\caL_{a_i+b_i}(\widetilde{X}_i)$ is still positive. In fact, since $m \geq i+3$,
	\begin{align*}
		\textit{exp}.\dim\caL_{a_i+b_i}(\widetilde{X}_i) & = \textit{exp}.\dim\caL_{a_i+b_i}({X}_i) - (m-3-i) = \\ 
		& = \left(1+i(m-\frac{7}{2}) - \frac{1}{2}i^2\right) - (m-3-i) \geq {i \choose 2} + 1.
	\end{align*}
	 Now, for a general element $C' \in \caL_{a_i+b_i}(\widetilde{X}_i)$.
	\begin{align*}
		\deg(C\cap C') = a_i(2m-4) & + b_i + (4m-7)m_i + m - 3 - i = \\
		& = (2m-3)(a_i+b_i) + 1 = \deg(C)\deg(C') + 1.
	\end{align*}
	Hence, the curve $C$ is a fixed component and can be removed. Then, by induction, we have
	\begin{align*}
		\dim\caL_{a_i+b_i}(X_i) \leq & \dim\caL_{a_i+b_i}(\widetilde{X}_i) + (m - 3 - i) = \dim\caL_{a_{i-1}+b_{i-1}}(X_{i-1}) + (m-3-i) = \\ & = (a_{i-1}+1)(b_{i-1}+1) - (4m-7){m_{i-1} + 1 \choose 2} + (m-3-i) = \\
		& = \textit{exp}.\dim\caL_{a_i+b_i}(X_i).
	\end{align*}
	This concludes the proof.
\end{proof}
\bibliographystyle{alpha}
\bibliography{references}
\end{document}